\documentclass[preprint,11pt]{elsarticle}

\usepackage{geometry}
\usepackage{amssymb}
\usepackage[utf8]{inputenc}
\usepackage[english]{babel}
\usepackage{amsmath, amstext, amsfonts, mathrsfs,amsthm}
\usepackage{amsfonts}
\usepackage{amssymb}
\usepackage{dsfont}
\usepackage{mathrsfs}
\usepackage{graphicx}
\usepackage{epsfig}
\usepackage{color}
\usepackage{bm}
\usepackage{verbatim}
\usepackage{stmaryrd}
\usepackage[colorlinks=false,linkcolor=darkblue]{hyperref}
\usepackage{bbm}			
\usepackage{enumerate} 
\usepackage{nicefrac}
\usepackage{xfrac}
\usepackage{units} 
\usepackage{mathtools}
\usepackage{pdfcomment}
\usepackage{url}
\usepackage{tikz}
\usepackage[labelfont={bf,sf},font={small},%
  labelsep=space]{caption}
\usepackage{acronym}
\usepackage{mathrsfs}
\usepackage{fancyhdr}
\usepackage[T1]{fontenc}
\usepackage{caption, booktabs}
\usepackage{textcomp}
\usepackage{multirow}

\usepackage{tikz}
\usetikzlibrary{positioning,arrows}
\usetikzlibrary{matrix}
\usepackage{hologo}
\usepackage{xcolor}
\usetikzlibrary{matrix,fit}
\tikzset{%
  mleftdelimiter/.style={inner ysep=0pt, inner xsep=1ex,left delimiter=\{,label={[label distance=3mm]left:#1}}
}

\usepackage[Bjarne]{fncychap} 

\usepackage{setspace}
\usepackage[squaren]{SIunits}

\usepackage{overpic}
\usepackage{subcaption}

\newcommand{\N}{\mathbb{N}}

\newcommand{\R}{\mathbb{R}}
\newcommand{\Q}{\mathbb{Q}}

\newcommand{\cP}{\mathcal{P}}

\newcommand{\cC}{\mathcal{C}}

\newcommand{\cF}{\mathcal{F}}
\newcommand{\cG}{\mathcal{G}}
\newcommand{\cH}{\mathcal{H}}

\newcommand{\cA}{\mathcal{A}}
\newcommand{\cB}{\mathcal{B}}

\newcommand{\cQ}{\mathcal{Q}}
\newcommand{\cS}{\mathcal{S}}

\newcommand{\1}{\mathds{1}}

\newcommand{\de}{\mathrm{\,d}}

\newcommand{\V}{\mathrm{Var_{V}}}
\newcommand{\Var}{\mathrm{Var_{HK}}}

\newcommand{\setcounterprefix}[2]{%
  \setcounter{#1}{0}%
  \expandafter\def\csname theH#1\endcsname{#2.\arabic{#1}}
  \expandafter\def\csname the#1\endcsname{#2.\arabic{#1}}%
}

\DeclareMathOperator{\v@r}{V@R}

\DeclareMathOperator{\av@r}{AV@R}

\newtheorem{theorem}{Theorem}[section]
\newtheorem{proposition}[theorem]{Proposition}
\newtheorem{assumption}[theorem]{Assumption}
\newtheorem{corollary}[theorem]{Corollary}
\newtheorem{definition}[theorem]{Definition}
\newtheorem{lemma}[theorem]{Lemma}
\newtheorem{example}[theorem]{Example}
\newtheorem{remark}[theorem]{Remark}

\begin{document}

\begin{frontmatter}



\title{On a version of a multivariate integration by parts formula \\ for Lebesgue integrals}


\author[label1]{Jonathan Ansari}

\affiliation[label1]{organization={Department of Artificial Intelligence and Human Interfaces \\ Paris Lodron Universität Salzburg},
            country={Austria},
            }

%

\begin{abstract}
Multidimensional integration by parts formulas apply under the standard assumption that one of the functions is continuous and the other has bounded Hardy-Krause variation.
Motivated by recently developed results in the probabilistic context of price and risk bounds, this paper provides a version of an integration by parts formula for the Lebesgue integral of measure-inducing functions which may both be discontinuous and may have infinite Hardy-Krause variation.
To this end,  we give a general definition of measure-inducing functions and establish various of their properties,  such as a characterization in terms of \(\Delta\)-monotone functions.
As a consequence of the integration by parts formula, several convergence results are provided, allowing an extension of the Lebesgue integral of a measure-inducing function to the case where one integrates with respect to a continuous semi-copula. 
The latter class of aggregation functions includes quasi-copulas which serve as bounds for the dependence structure in many applications.
\end{abstract}

%

\begin{keyword}
capacity
\sep
\(\Delta\)-monotone function
\sep
Hardy-Krause variation
\sep
measure-inducing function
\sep
quasi-copula
\sep
semi-copula 


\end{keyword}

\end{frontmatter}




\section{Introduction}\label{intro}

Integration by parts is a standard tool in analysis. In the one-dimensional case, well-known conditions on functions \(f,g\colon [0,1]\to \R\), for which a formula
\begin{align}\label{IP1dim}
\int_0^1 f(x) \de g(x) = f(1)g(1)-f(0)g(0)-\int_0^1 g(x) \de f(x)
\end{align}
is valid, are continuity of one of the functions and bounded variation of the other one. A characterization of the convergence of the above Stieltjes integrals and, thus, for formula \eqref{IP1dim} to hold true can be found in \cite[Theorems II.11.7 and II.13.17] {Hildebrandt-1963}.
A multivariate extension of \eqref{IP1dim} for Stieltjes integrals is based on Abel's transform and goes back to \cite{Young-1917}. Then, if one of the functions is continuous, a necessary and sufficient condition for the convergence of the integrals is that the other function has bounded Hardy-Krause variation, which also accounts for bounded variation of the function in the sense of Vitali on all lower dimensional faces of the domain, see \cite[Proposition 2]{Zaremba-1968}.
We refer to  \cite{Appell-2014,Adams-1933} for a survey on several concepts of bounded variation and to \cite{Aistleitner-2015,Aistleitner-2017,Basu-2016,Hardy-1905,
Krause-1903,Owen-2005} for the Hardy-Krause variation. 
In the two-dimensional case, for functions \(f,g\colon [0,1]^2\to \R\,,\) one of several variants of an integration by parts formula is
\begin{align*}
\int_{0,0}^{1,1} f(x,y)\de g(x,y) = f(0,0) \big[g\big]_{0,0}^{1,1} +\int_{0,0}^{1,1} \big[g\big]_{x,y}^{1,1} \de f(x,y) + \int_0^1 \big[g\big]_{x,0}^{1,1} \de f(x,0) + \int_0^1 \big[g\big]_{0,y}^{1,1} \de f(0,y)\,,
\end{align*}
where \(\big[g\big]_{a_1,a_2}^{b_1,b_2} := g(a_1,b_1)-g(a_1,b_2)-g(a_2,b_1)+g(a_2,b_2)\,,\) see \cite{Young-1917}. If \(g\) is grounded, i.e., if \(g(x,y)=0\) whenever \(x=0\) or \(y=0\,,\) then the right-hand side of the above formula simplifies to
\begin{align*}
 f(0,0) g(1,1) +\int_{0,0}^{1,1} \big[g\big]_{x,y}^{1,1} \de f(x,y) + \int_0^1 [g(1,1)-g(x,1)] \de f(x,0) + \int_0^1 [g(1,1)-g(1,y)] \de f(0,y)\,.
\end{align*}
Recent developments in probability motivate a version of an integration by parts formula for the Lebesgue integral on a possibly non-compact domain \(\Xi=\bigtimes_{i=1}^d \Xi_i\) for intervals \(\Xi_1,\ldots,\Xi_d\subseteq \R\), where the integrator \(g\colon \Xi \to \R\) is a right-continuous, bounded and grounded function, for example a \(d\)-variate distribution function or a composition \(g=Q\circ(F_1,\ldots,F_d)\) of a quasi-copula \(Q\) with univariate distribution functions \(F_i\colon \Xi_i\to [0,1]\,,\) \(i\in \{1,\ldots,d\}\,.\)
Several applications in financial mathematics can be found in the context of price bounds in a model-free setting under incorporating market-implied dependence information like knowledge of correlations or option prices \cite{Ansari-2021,Lux-2017,Yanez-2020,Tankov-2011}. In such settings,  the distributional bounds are often given by quasi-copulas, which generally do not induce a signed measure 
 \cite{Durante-2016,Nelsen-2011,Nelsen-2006,Nelsen-2004,Nelsen-2010}.
Further, it is desirable that the integrand \(f\colon \Xi \to \R\) is not assumed to be continuous, for example, when one determines improved price bounds for an option with a discontinuous payoff function.

The following main result provides general conditions on possibly discontinuous measure-inducing functions \(f,g\colon \Xi\to \R\) for which an integration by parts formula is valid. For \(I\subseteq\{1,\ldots,d\}\,,\) the lower \(I\)-marginal of a function \(h\) is denoted by \(h_I\), see \eqref{defqmargl}, 
the integrator \(\de \nu_h\) refers to the Lebesgue integral with respect to the signed measure \(\nu_h\) induced by the function \(h\) due to \eqref{defmind},
and \(\overline{h}\) is the survival function associated with \(h\,,\) see \eqref{defsurvfun}. The classes \(\cF_{mi}^c(\Xi)\) and \(\cF_{mi}^{c,l}(\Xi)\) of measure-inducing functions are defined in \eqref{defclassmi} and \eqref{defclassmif}. Note that the concept of a measure-inducing function in \eqref{defmind} is different from the familiar notion of a measure-generating function, see \eqref{defmgfun}.


%

\begin{theorem}[Multivariate integration by parts]\label{proppif}~\\
For \(f \in \cF_{mi}^{c,l}(\Xi)\) and \(g\in \cF_{mi}^c(\Xi)\,,\) assume that \(f_\emptyset\) exists and that \(g\) is bounded and grounded. 
If \(f\) and \(g\) have no common discontinuities on the same side of each point, then
\begin{align}\label{genformi1}
\int_\Xi f(x)\de \nu_g = \sum_{I\subseteq \{1,\ldots,d\}\atop I\ne \emptyset} 
\int_{\Xi_I} \overline{g}_I(u)\de \nu_{f_I}(u)
+ \overline{g}_\emptyset f_\emptyset\,,
\end{align}
whenever all integrals exist.
\end{theorem}
The functions \(f\) and \(g\) in the above theorem are assumed to have no common discontinuities on the same side of each point (see \eqref{defcomdiscy} for the definition) which is satisfied, for example, if one of the functions is continuous or if one is left-continuous and the other right-continuous. In the case where \(f\) and \(g\) do have common discontinuities on the same side of some points, an additional term respecting the jumps has to be included, see Example \ref{exajumdis}. 

Under some integrability conditions on the function \(g\,,\) the right-hand side of \eqref{genformi1} is closed with respect to pointwise convergence of the integrands. Hence, Theorem \ref{proppif} allows an extension of the Lebesgue integral \(\int_\Xi f \de \nu_g\) to the case where \(g\) does not induce a signed measure, for example, when \(g\) is a quasi-copula or, more generally, a (Lipschitz-)continuous semi-copula; see Theorems \ref{theconvfin} and \ref{theconvthsec}, which can also be interpreted as a version of a dominated convergence theorem for the integrator.
In particular, formula \eqref{genformi1} implies a version of the standard integration by parts formula for Stieltjes integrals where \(\Xi=[0,1]^d\) is compact and where one of the functions is continuous and the other has bounded Hardy-Krause variation.
Note that for a non-compact domain \(\Xi\), the function \(f\) in Theorem \ref{proppif} may also have unbounded Hardy-Krause variation.

Dispensing with the continuity assumption, an integration by parts formula for functions in two variables is given for weak net integrals in \cite[Theorem III.8.8]{Hildebrandt-1963}, see \cite[Proof of Theorem 3b]{Rueschendorf-1980} for an application of this formula to stochastic orderings. Some general sufficient conditions for the convergence of weak net integrals can be found in \cite[Section III.8]{Hildebrandt-1963} in terms of net integrals.
Several versions of two-dimensional integration by parts formulas for Lebesgue integrals where both the integrator and the integrand induce positive measures can be found in \cite{Durante-2016,Nelsen-2007,Li-2002,Tankov-2011}. An extension to signed measures is given for the bivariate case in \cite{Berghaus-2017} and for the multidimensional case in \cite{Radulovic-2017} for finite signed measures and in \cite{Lux-2017b,Lux-2017,Yanez-2020} for not necessarily finite signed measures. In the latter three references, however, the construction is not precise and their formulas do not apply under the general assumptions claimed there.

In this paper, we fix these inaccuracies and provide a general framework for an integration by parts formula for Lebesgue integrals. The underlying functions are defined on a possibly unbounded domain and are not assumed to be continuous. As motivated above, we assume that the integrator \(g\) is grounded and bounded.
In Section \ref{secmif}, we give a precise definition of a measure-inducing function and outline their basic properties that are used for the main results of the paper. For a detailed study of measure-inducing functions and for several examples, we refer to Appendix A.
In Sections \ref{secmr} and \ref{secvonv}, we provide the main contributions of the paper, that are, an integration by parts formula for measure-inducing functions, a transformation formula for the integration by parts operator, and several convergence results that allow an extension of the Lebesgue integral of a measure-inducing function to the case where one integrates w.r.t. a function that does not necessarily induce a signed measure like a Lipschitz-continuous semi-copula. All proofs are transferred to Appendix B.

%
%
%
%
%

\section{Measure-inducing functions}\label{secmif}

We begin with a general definition of a measure-inducing function. Then we provide the necessary notation and outline several properties of measure-inducing functions, which are needed for the integration by parts formulas in \eqref{genformi1} and in Sections \ref{secmr} and \ref{secvonv}. 

%
%
%
%
%
%
%
%
%

\subsection{Definition of a measure-inducing function}

It is well-known that every right-continuous, non-decreasing function \(F\colon \R\to[0,1]\) with \(\lim_{x\to -\infty} F(x)=0\) and \(\lim_{x\to \infty} F(x)=1\) induces a uniquely determined probability measure \(P\) on the Borel \(\sigma\)-algebra of \(\R\) such that \(P((x,y])=F(y)-F(x)\) for all \(x,y\in \R\) with \(x\leq y\,.\) Dispensing with the assumption of right-continuity, \(P\) is still uniquely determined through \(F\) by
\begin{align}\label{deffP}
P([x,y])=\lim_{y'\downarrow y}F(y') - \lim_{x'\uparrow x} F(x')\,, ~~~ \text{for all } x,y\in \R \text{ such that } x\leq y\,,
\end{align}
i.e., the value of each closed interval is approximated from the outside.
If \(F\) has bounded variation and is no more assumed to be non-decreasing, then, in general, \(P\) defined by \eqref{deffP} is a finite signed measure. 
However, if \(F\) has unbounded variation (but is still bounded), it does not necessarily induce a signed measure by \eqref{deffP}.

For deriving a multidimensional integration by parts formula for Lebesgue integrals in Section \ref{secmr}, we extend the above considerations and study real-valued, multivariate functions that induce signed measures in a general setting as follows.\\

For \(d\in \N\,,\) recall that \(\Xi=\bigtimes_{i=1}^d \Xi_i\) denotes the Cartesian product of the non-empty intervals \(\Xi_1,\ldots,\Xi_d\subseteq \R\,.\) For \(I\subset \{1,\ldots,d\}\,,\) \(I\ne \emptyset\,,\) denote by \(\Xi_I:=\bigtimes_{i\in I} \Xi_i\) the Cartesian product of \((\Xi_i)_{i\in I}\,.\) For \(i\in \{1,\ldots,d\}\,,\) we denote the infimum and supremum of \(\Xi_i\) by 
\begin{align*}
a_i&:=\inf \{\Xi_i\}, && b_i:=\sup\{\Xi_i\}\,,
\end{align*}
respectively, where both \(a_i\) and \(b_i\) may attain a finite or infinite value. Typically, we choose for \(\Xi_i\) intervals like \([0,1]\,,\) \([0,1)\,,\) \((0,1)\,,\) \(\R\,,\) or \(\R_+:=[0,\infty)\,.\) We set \(a:=(a_1,\ldots,a_d)\) and \(b:=(b_1,\ldots,b_d)\,.\)
For \(x=(x_1,\ldots,x_d),y=(y_1,\ldots,y_d)\in \R^d\,,\) write \(x\leq y\) whenever \(x_i\leq y_i\) for all \(i\in \{1,\ldots,d\}\,.\) Further, denote for \(x\leq y\) by \([x,y]:=\bigtimes_{i=1}^d [x_i,y_i]\) the closed cuboid spanned by \(x\) and \(y\) and, similarly, by \([x,y)\,,\) \((x,y]\,,\) and \((x,y)\) the respective componentwise half-open and open cuboid.


Denote by \(\cB(\Xi)\) the Borel \(\sigma\)-algebra on \(\Xi\,.\)
A signed measure on \((\Xi,\cB(\Xi))\) is a countably additive set function \(\nu\colon \cB(\Xi)\to \R\cup\{\pm\infty\}\) which attains at most one of the values \(+\infty\) and \(-\infty\) and fulfils \(\nu(\emptyset)=0\,.\) By the Jordan decomposition theorem, there exist two uniquely determined measures on \((\Xi,\cB(\Xi))\,,\) the \emph{positive part} \(\nu^+\) and the \emph{negative part} \(\nu^-\,,\) such that \(\nu=\nu^+-\nu^-\) where at least one of the measures \(\nu^+\) and \(\nu^-\) is finite, see, e.g., \cite[Chapter 5.1]{Benedetto-2009}. The tuple \((\nu^+,\nu^-)\) is called the \emph{Jordan decomposition} of \(\nu\,.\)
Similar to the one-dimensional case in \eqref{deffP}, the value of a closed cuboid with respect to a signed measure \(\nu\) can be approximated from the outside by using some standard arguments for measures. Using such an approximation, we will then define a measure-inducing function without any continuity assumptions. This requires some notation, which we introduce as follows.

Let \(x,y\in \Xi\,.\)
Whenever there exist \(\delta=(\delta_1,\ldots,\delta_d)\,,\) \(\varepsilon=(\varepsilon_1,\ldots,\varepsilon_d)\) with \(\delta_i,\varepsilon_i>0\) for all \(i\in \{1,\ldots,d\}\) such that \(x-\delta,y+\varepsilon\in \Xi\) and \(\nu([x-\delta,y+\varepsilon])\) is finite, then the iterated limit
\begin{align}\label{eqappme}
\begin{split}
\lim_{\delta_i,\varepsilon_i \downarrow 0\atop i\in \{1,\ldots,d\}} \nu\bigg(\bigtimes_{i=1}^d [x_i-\delta_i,y_i+\varepsilon_i]\bigg)
&:=\lim_{\delta_{i_1}\downarrow 0} \lim_{\varepsilon_{j_1}\downarrow 0} \cdots  \lim_{\delta_{i_d} \downarrow 0} \lim_{\varepsilon_{j_d}\downarrow 0} \nu\bigg(\bigtimes_{i=1}^d [x_i-\delta_i,y_i+\varepsilon_i]\bigg) \\
&\phantom{:}= \nu\bigg(\bigcap_{\delta_i,\varepsilon_i\downarrow 0\atop i=1,\ldots,d} \bigtimes_{i=1}^d [x_i-\delta_i,y_i+\varepsilon_i]\bigg)\\
&\phantom{:}= \nu\left([x_1,y_1]\times \cdots \times [x_d,y_d]\right)
\end{split}
\end{align}
exists for all arrangements \((i_1,\ldots,i_d),(j_1,\ldots,j_d)\in S_d\) of the limits due to the continuity of measures, see, e.g. \cite[Theorem 2.4.3]{Benedetto-2009}, where \(S_d\) denotes the set of all permutations of \(\{1,\ldots,d\}\,.\)
For a function \(f\colon \Xi\to \R\,,\) for \(z=(z_1,\ldots,z_d)\in \Xi\,,\) for \(i\in \{1,\ldots,d\}\,,\) and for \(\varepsilon \geq 0\) such that \(z_i+\varepsilon\in \Xi\,,\)
we define the difference operator of length \(\varepsilon\) applied to the \(i\)th component of \(f\) by
\begin{align}\label{defdopfh}
\Delta_{\varepsilon}^i f(z):= f(z+\varepsilon e_i)-f(z)\,,
\end{align}
 where \(e_i\) denotes the \(i\)th unit vector.
Similarly, for \(x=(x_1,\ldots,x_d),y=(y_1,\ldots,y_d)\in \Xi\) with \(x\leq y\,,\) the \(d\)-variate difference operator is defined by
\begin{align}\label{defdvdo}
\Delta_{x,y}[f]:= \Delta_{\varepsilon_1}^1\cdots \Delta_{\varepsilon_d}^d f(x)\,, ~~~ \varepsilon_i:=y_i-x_i\,, ~~~ i\in \{1,\ldots,d\}\,.
\end{align}
Due to \eqref{eqappme} and the disintegration theorem (cf. \cite[Theorem 6.4]{Kallenberg-2002}),
\begin{align*}
\nu([x,y]) = \lim_{\delta_i,\varepsilon_i \downarrow 0\atop i\in \{1,\ldots,d\}}
\int_{x_1-\delta_1}^{y_1+\varepsilon_1} \cdots \int_{x_d-\delta_d}^{y_d+\varepsilon_d} \nu_d(\de z_d|z_{d-1},\ldots,z_1) \cdots \nu_1(\de z_1)
\end{align*}
depends on a \(d\)-variate function \(F\) via the difference operator \(\Delta_{x-\delta,y+\varepsilon}[F]\) for \(\delta,\varepsilon\downarrow 0\,,\) whenever \(\nu([x-\delta,y+\varepsilon])\) is finite, where \(\nu_d(\cdot| z_{d-1},\ldots,z_1)\) denotes the conditional signed marginal measure \(\nu\) with respect to the \(d\)th component conditional on \(z_{d-1},\ldots,z_1\,,\) and where \(\nu_1\) denote the signed marginal measure of \(\nu\) with respect to the first component.
In the one-dimensional case, the above expression yields the correspondence in \eqref{deffP}; in the two-dimensional case, one has
\begin{align}
\nonumber \nu([x,y])&=  \lim_{\delta_i,\varepsilon_i\downarrow 0\atop i\in \{1,2\}} \Delta_{x-\delta,y+\varepsilon} [F]\\
\label{lim1}  &= \lim_{\delta_1\downarrow 0} \lim_{\delta_2\downarrow 0} \lim_{\varepsilon_1\downarrow 0} \lim_{\varepsilon_2\downarrow 0}\Big[ F(y_1+\varepsilon_1,y_2+\varepsilon_2) - F(y_1+\varepsilon_1,x_2-\delta_2)\\
\nonumber&  ~~~~~~~~~~~~~~~~~~~~~~~~~ - F(x_1-\delta_1,x_2+\varepsilon_2) + F(x_1-\delta_1,x_2-\delta_2)\Big] 
\end{align}
independent of the arrangements of the limits in \eqref{lim1}.
For a general definition of a measure-inducing function,  we extend the differences considered above to arbitrary dimension. To this end, denote
by \(u\vee v:=(\max\{u_1,v_1\},\ldots,\max\{u_d,v_d\})\) and \(u\wedge v:=(\min\{u_1,v_1\},\ldots,\min\{u_d,v_d\})\) the componentwise maximum and minimum, respectively, where \(u=(u_1,\ldots,u_d),v=(v_1,\ldots,v_d)\in \R^d\,.\)
Then define for a function \(g\colon \Xi\to \R\) and \(x,y\in \Xi\) with \(x\leq y\) the \emph{limit difference operator}
\begin{align}\label{defitlimi}
D_{x,y}[g]:= \lim_{\delta_i,\varepsilon_i \downarrow 0\atop i\in \{1,\ldots,d\}} \Delta_{a\vee (x-\delta),b\wedge (y+\varepsilon)} [g] = \lim_{\delta_{1}\downarrow 0} \lim_{\varepsilon_{1}\downarrow 0} \cdots  \lim_{\delta_{d} \downarrow 0} \lim_{\varepsilon_{d}\downarrow 0}  \Delta_{a\vee (x-\delta),b\wedge (y+\varepsilon)} [g] \,,
\end{align}
whenever the iterated limit on the right-hand side exists and coincides for all permutations of the limits. If \(x\) (and similarly for \(y\)) lies on the boundary of \(\Xi\) in the \(i\)-th coordinate, i.e., if \(a_i\in \Xi_i\) and \(x_i=a_i\,,\) then the limit \(\lim_{\delta_i\downarrow 0}\) on the right-hand side of \eqref{defitlimi} is superfluous because \(g\) is evaluated in the \(i\)-th coordinate at \(a_i\vee (x_i-\delta_i) = a_i=x_i\,.\)\\


Since a (signed) measure is uniquely determined on the set of closed cuboids, the above considerations motivate to generally define a measure-inducing function by the limit difference operator in \eqref{defitlimi} as follows.


\begin{definition}[Measure-inducing function]\label{defmifun}~\\
A function \(f\colon \Xi \to \R\) is said to be \emph{measure-inducing} if there exists a signed measure \(\nu\) such that
\begin{align}\label{defmind}
\nu([x,y]) = D_{x,y} [f] 
\end{align}
for all \(x,y \in \Xi\) with \(x\leq y\,.\) 
Denote by \(\nu_f\) the signed measure induced by \(f\,.\)
\end{definition}

By definition of a measure-inducing function, the componentwise left- and right-hand limits do not necessarily exist, see Example \ref{exmindgb}\eqref{exmindgba}.
Further, the mass \(\nu_f(\{x\})\) at \(x\) does not depend on \(f(x)\,.\)
To avoid technical difficulties, we often consider a subclass of measure-inducing functions on \(\Xi\), making the following assumptions, which are typically fulfilled in applications. We denote by \(\overset{\circ}{\Xi}\) the set of inner points of \(\Xi\,.\)
\begin{assumption}[Componentwise limits; continuity at the boundary of \(\Xi\)]\label{asscont}~
\begin{enumerate}[(i)]
\item \label{asscont1} For \(f\colon \Xi \to \R\,,\) the componentwise left- and right-hand limits exist. Further, for all \(x\in \overset{\circ}{\Xi}\,,\) there exist \((i_1,\ldots,i_d)\in S_d\) and \(k_j\in \{-1,1\}\) such that 
\begin{align}\label{contcond0}
f(x)=\lim_{\varepsilon_1\downarrow 0} \cdots \lim_{\varepsilon_d\downarrow 0} f\big(x+\sum_{j=1}^d k_j\varepsilon_j e_{i_j}\big)\,.
\end{align}
\item \label{asscont2} \(f\colon \Xi \to \R\) is continuous at the boundary of \(\Xi\,,\) i.e., whenever \(a_i=\inf\{\Xi_i\}\in \Xi_i\) and \(b_i=\sup\{\Xi_i\}\in \Xi_i\,,\) respectively, for \(i\in \{1,\ldots,d\}\,,\) the function \(f\) satisfies
\begin{align}\label{contboun1}
\lim_{h\downarrow 0} f(x_1,\ldots,x_{i-1},a_i+h,x_{i+1},\ldots,x_d) &= f(x_1,\ldots,x_{i-1},a_i,x_{i+1},\ldots,x_d), ~~ \forall x_j\in \Xi_j\,, j\ne i\,, \\
\label{contboun2}\lim_{h\downarrow 0} f(x_1,\ldots,x_{i-1},b_i-h,x_{i+1},\ldots,x_d) &= f(x_1,\ldots,x_{i-1},b_i,x_{i+1},\ldots,x_d), ~~ \forall x_j\in \Xi_j\,, j\ne i.
\end{align}
\end{enumerate}
\end{assumption}
Assumption \ref{asscont}\eqref{asscont1} states that the value \(f(x)\) is given by one of the componentwise limits of \(f\) at \(x\,,\) which all are assumed to exist. Assumption \ref{asscont}\eqref{asscont2} is continuity of \(f\) at the boundary of \(\Xi\,.\)
Denote by
\begin{align}
\nonumber \cF_{mi}(\Xi)&:=\left\{ f\colon \Xi \to \R \mid f\text{ is measure-inducing and satifies Assumption \ref{asscont}\eqref{asscont1}}\right\}\,,\\
\label{defclassmi}\cF_{mi}^c(\Xi)&:=\left\{ f\colon \Xi \to \R \mid f\text{ is measure-inducing and satifies Assumptions \ref{asscont}\eqref{asscont1} and \eqref{asscont2}}\right\}
\end{align}
the class of measure-inducing functions on \(\Xi\) satisfying Assumption \ref{asscont}\eqref{asscont1} (and \eqref{asscont2}).
Note that \(\cF_{mi}(\Xi)=\cF_{mi}^c(\Xi)\) if \(\Xi=(a,b)\) because, in this case, Assumption \ref{asscont}\eqref{asscont2} is trivially satisfied.


\begin{remark}\label{remindmeas}
\begin{enumerate}[(a)]
\item \label{remindmeas1} If the domain \(\Xi\) is a compact subset of \(\R^d\,,\) i.e., if \(\Xi=[a,b]\subset \R^d\,,\) then the induced signed measure \(\nu_f\)
is finite. This follows from
\begin{align}\label{eqremindmeas1}
\nu_f^+([a,b])-\nu_f^-([a,b])=\nu_f([a,b])=D_{a,b}[f] = \Delta_{a,b}[f]\in \R
\end{align}
using that, due to the Jordan decomposition of a signed measure, at least one of the measures \(\nu_f^+\) and \(\nu_f^-\) is finite. This yields, in particular, that a measure-inducing function always induces a \(\sigma\)-finite measure. Conversely, due to Definition \ref{defmifun}, \(f\) can only induce a non-finite signed measure by \eqref{defmind} if the domain \(\Xi\) is a non-compact subset of \(\R\,.\)
\item \label{eqremindmeas2a} If a measure-inducing function \(f\colon \Xi\to \R\) is bounded, then \(\nu_f\) is finite. This follows for \(\Xi=[a,b]\) from \eqref{eqremindmeas1}. For \(\Xi=(a,b)\,,\) it holds that
\begin{align*}
\nu_f(\Xi) = \lim_{x_i\downarrow a_i, y_i\uparrow b_i\atop i=1,\ldots,d} \nu_f([x,y]) = \lim_{x_i\downarrow a_i, y_i\uparrow b_i\atop i=1,\ldots,d}  D_{x,y}[f] \leq 2^d M\,,
\end{align*}
where \(x=(x_1,\ldots,x_d),y=(y_1,\ldots,y_d)\in (a,b)\) and where \(M\in \R_+\) is an upper bound for \(|f|\,.\) The other cases where \(\Xi\) is a product of closed, half-open, and open intervals follow similarly.
\item \label{eqremindmeas2} From the definition of the operator \(\Delta_{x,y}[f]\,,\) it follows that \(\Delta_{x,y}[f]=0\) whenever \(x_i=y_i\) 
for a component \(i\in \{1,\ldots,d\}\,.\) Further, if \(f\) is measure-inducing and constant in at least one component, then \(f\) induces the null measure, i.e., \(\nu_f(A)=0\) for all Borel sets \(A\subseteq \Xi\,.\)
\item \label{eqremindmeas3} The continuity assumptions \eqref{contcond0} - \eqref{contboun2} are mild. They allow \(\nu_f\) having point mass at inner points of \(\Xi\) and only exclude the case that \(\nu_f\) has point mass at the boundary of \(\Xi\,.\) To see the latter, assume w.l.o.g. that \(a_1\in \Xi_1\,.\) Then, for \(x=(a_1,w)\in \Xi\,,\) it holds by \eqref{contboun1} that
\begin{align*}
\nu_f(\{x\})= D_{x,x}[f] &= \lim_{\varepsilon\downarrow 0} D_{w,w} [f(a_1+\varepsilon,\cdot)] - D_{w,w}[f(a_1,\cdot)] \\&= D_{w,w}[\lim_{\varepsilon \downarrow 0 } [f(a_1+\varepsilon,\cdot) - f(a_1,\cdot)] =  0\,.
\end{align*}
 However, extending \(f\) to a larger domain \(\Xi'\supset \Xi\,,\) it is possible to consider the case where the induced signed measure of \(f|_{\Xi}\) has finite point mass also at the boundary of \(\Xi\,.\)
\end{enumerate}
\end{remark}


\subsection{Properties of measure-inducing functions}

In the sequel, we outline various properties of measure-inducing functions, which are used for the integration by parts formula in Section \ref{secmr}. A detailed study of measure-inducing functions concerning left- and right-continuous versions, transformations, marginal functions, the Hardy-Krause variation, the Jordan decomposition, and survival functions is given in Appendix A.

\subsubsection{Left-/right-continuity}

If \(f\in \cF_{mi}(\Xi)\) is left-/right-continuous, then the induced signed measure \(\nu_f\) can be expressed by the difference operator \(\Delta_{x,y}[f]\) in \eqref{defdvdo} through
\begin{align*}
\nu_f\left([x,y)\right) &= \Delta_{x,y}[f] && \text{and} &&  \nu_f\left((x,y]\right) = \Delta_{x,y}[f],
\end{align*}
respectively, see Proposition \ref{lemfmi}. If \(f\) additionally satisfies the continuity conditions \eqref{contboun1} and \eqref{contboun2} at the boundary of \(\Xi\,,\) then the left- and right-continuous versions \(f_-\) and \(f_+\) of \(f\) induce the same signed measure, i.e.,
\begin{align*}
\nu_f = \nu_{f_-} = \nu_{f_+}\,,
\end{align*} 
see Lemma \ref{lemunimifun} and Example \ref{exlemunimifun}.

\subsubsection{Marginal function and marginal measure}

Formula \eqref{genformi1} uses the concept of marginal functions.
For \(I=\{i_1,\ldots,i_k\}\subseteq \{1,\ldots,d\}\,,\) \(I\ne \emptyset\,,\) the \emph{lower \(I\)-marginal} \(f_I\) of a function \(f\colon \Xi\to \R\) is defined as the function 
\begin{align}\label{defqmargl}
f_I(x_{i_1},\ldots,x_{i_k}) := \lim_{x_j\downarrow a_j\atop j\notin I} f(x_1,\ldots,x_d)\,, ~~~ x_{i_j}\in \Xi_{i_j}\,, ~j\in \{1,\ldots,k\}\,,
\end{align}
 whenever the iterated limit exists independent of the arrangement of the limits. The \emph{upper \(I\)-marginal function} \(f^I\) as well as the constants \(f_\emptyset\) and \(f^\emptyset\) are defined analogously in \eqref{defqmarg} and \eqref{eqfmess}.
For \(I=\{1,\ldots,d\}\,,\) one has \(f_I=f=f^I\,.\)

Since \(I\)-marginal functions of a measure-inducing function are generally not measure-inducing, see Example \ref{exmindgb}, we consider in Theorem \ref{proppif} the class
\begin{align}
\label{defclassmif}
\cF_{mi}^{c,l}(\Xi)&:=\left\{ f\in \cF_{mi}^c\mid f_I\in \cF_{mi}^c(\Xi_I) \text{ for all } I\subset \{1,\ldots,d\}\,, I\ne \emptyset\right\}\,,
\end{align}
of measure-inducing functions whose lower \(I\)-marginal functions are also measure-inducing and continuous at the boundary of \(\Xi_I\,.\)
As shown in Theorem \ref{propbhkv}, functions in the class \(\cF_{mi}^{c,l}([0,1]^d)\) are characterized by bounded Hardy-Krause variation, which implies for compact domain \(\Xi\) and continuous \(f\) or \(g\) that the integration by parts formula \eqref{genformi1} is closely related to the standard integration by parts formulas for Stieltjes integrals, see, e.g., \cite[Proposition 2]{Zaremba-1968}.
However, for a non-compact domain \(\Xi\,,\) functions in the class \(\cF_{mi}^{c,l}(\Xi)\) are not necessarily bounded and thus they may have infinite Hardy-Krause variation.
A characterization of the class \(\cF_{mi}^{c,l}(\R_+^d)\) through \(\Delta\)-monotone functions is given in Theorem \ref{thecfmicl}. 
Using the transformation \eqref{eqpropmeatraf} for induced signed measures, characterizations for a general domain \(\Xi\) can be obtained similarly.

The \(I\)-marginal measure \(\nu^I\) of a signed measure \(\nu\) is defined as usual by \eqref{defImargmeas}. Note that, in general, the signed measure \(\nu_{f_I}\) induced by the lower \(I\)-marginal of \(f\in \cF_{mi}^{c,l}(\Xi)\), as used in formula \eqref{genformi1}, does not coincide with the \(I\)-marginal measure induced by \(f\,,\) i.e., \(\nu_{f_I}\ne \nu_f^I\) in general, see Example \ref{exmmamibl}\eqref{exmmamibl1}.

\subsubsection{Grounded functions}

In probabilistic applications, an important property of aggregation functions is that they vanish whenever a component tends to the lower boundary of the domain. This property is often denoted as \emph{grounded}, which we define for a function \(f\colon \Xi\to \R\) that fulfils the continuity condition \eqref{contboun1} by 
\begin{align}\label{eqdefgrou}
\lim_{x_j\downarrow a_j} f(x_1,\ldots,x_d)=0 ~~~\text{for all } j\in \{1,\ldots,d\} \text{ and } x_i\in \Xi_i\,, i\in \{1,\ldots,d\}\setminus \{j\}\,,
\end{align}
cf. \cite{Nelsen-2006}.
Note that condition \eqref{eqdefgrou} is equivalent to \(f_I(x)=0\) for all \(x\in \Xi_I\) and \(I\subsetneq \{1,\ldots,d\}\,,\) i.e., all lower \(I\)-marginals of a grounded function \(f\) induce the null measure.
Some properties of grounded functions that are used in the proof of Theorem \ref{proppif} are given in Lemmas \ref{lemmargdisgrou} and \ref{lemsurvmarg}.

We make use of the following classes of aggregation functions on \([0,1]^d\,,\) which are, in particular, bounded and grounded functions and thus candidates for \(g\) in the integration by parts formulas \eqref{genformi1} and \eqref{genformi18}.

\begin{definition}[Semi-copula, quasi-copula, copula]\label{defsemcop}~
\begin{itemize}
\item[(a)] A (\(d\)-variate) \emph{semi-copula} is a function \(S\colon [0,1]^d\to [0,1]\) that
\begin{enumerate}[(i)]
\item \label{defsemcop1} has uniform marginals, i.e., for all \(i \in \{1,\ldots,d\}\,,\) it holds \(S(u)=u_i\) for all \(u=(u_1,\ldots,u_d)\in [0,1]^d\) whenever \(u_j=1\) for all \(j\ne i\,.\)
\item \label{defsemcop2} is componentwise increasing, i.e., \(S(u)\leq S(v)\) whenever \(u\leq v\,.\)
\end{enumerate}
\item[(b)] A (\(d\)-variate) \emph{\(L\)-Lipschitz semi-copula} is a (\(d\)-variate) semi-copula that satisfies for some \(L\geq 1\) the Lipschitz condition
\begin{enumerate}[(i)]\setcounter{enumi}{2}
\item \(|S(u)-S(v)|\leq L\, \sum_{i=1}^d |u_i-v_i|\) for all \(u=(u_1,\ldots,u_d),v=(v_1,\ldots,v_d)\in [0,1]^d\,.\)
\end{enumerate}
\item[(c)] A (\(d\)-variate) \emph{quasi-copula} is an \(1\)-Lipschitz semi-copula.
\item[(d)] A (\(d\)-variate) \emph{copula} is a function \(C\colon [0,1]^d\) that is grounded, \(d\)-increasing and fulfils \eqref{defsemcop1}.
\end{itemize}
Denote by \(\cS_d\,,\) \(\cS_{d,L}\,,\) \(\cQ_d\,,\) and \(\cC_d\)  the class of \(d\)-variate semi-copulas, \(L\)-Lipschitz semi-copulas, quasi-copulas, and copulas, respectively.
\end{definition}

Note that property \eqref{defsemcop1} and \eqref{defsemcop2} imply groundedness.
It holds that \(\cC_d \subset \cQ_d\subset \cS_{d,L}\subset \cS_d\,,\) where every inclusion is strict apart from the trivial cases \(d=1\) and \(L = 1\,.\) While copulas describe due to Sklar's theorem the dependence structure of probability distributions, semi-copulas relate to the dependence structure of capacities. Further, quasi-copulas have the property that they form a complete lattice, which however is only in the bivariate case order-isomorphic to the Dedekind–MacNeille completion of the poset of copulas, see, e.g., \cite{Durante-2006,Durante-2016,Nelsen-2006} for an overview of these concepts.

\subsubsection{Survival function}

The representation of the integration by parts formula \eqref{genformi1} uses the concept of survival functions: For a function \(f\colon \Xi \to \R\) that fulfils the continuity condition \eqref{contboun2}, the survival function \(\overline{f}\colon \Xi\to \R\) of \(f\)
is defined by
\begin{align}\label{defsurvfun}
\overline{f}(x):=\sum_{k=0}^d\sum_{I\subseteq \{1,\ldots,d\}\atop I=\{i_1,\ldots,i_k\}} (-1)^{|I|} f^I(x_{i_1},\ldots,x_{i_k})\,, ~~~x=(x_1,\ldots,x_k)\in \Xi\,,
\end{align}
whenever all upper \(I\)-marginals exist.
For example, if \(F\) is the cumulative distribution function of a random vector \(X=(X_1,\ldots,X_d)\) on a probability space \((\Omega,\mathcal{A},P)\,,\) then the survival function \(\overline{F}\) of \(F\) is given by
\begin{align}\label{survdisfun}
 \overline{F}(x) = \int \1_{\{z>x\}} \de F(z) = P(X_i> x_i ~\forall i\in \{1,\ldots,d\}) \,.
\end{align}
In general, taking the \(I\)-marginal and taking the survival function do not commute, i.e., \(\overline{f^I} \ne \overline{f}^I\) and \(\overline{f_I}\ne \overline{f}_I\,;\)
see Example \ref{exuimsf}. However, if \(f\) is grounded, then \(\overline{f^I}= \overline{f}_I\,;\) see Lemma \ref{lemsurvmarg}. This property is used in the proof of Theorem \ref{proppif}.

\section{Integration by parts for measure-inducing functions}\label{secmr}

This section presents the main results of integration by parts.
After introducing the expectation operator and the integration by parts operator, we formulate Theorem \ref{proppif} in this short notation and provide some direct implications.
In the second part of this section, we give a transformation formula for the integration by parts operator, which corresponds to the change of variables for Lebesgue integrals.

Let \(f\colon \Xi\to \R\) be a measurable function and let \(g\colon \Xi\to \R\) be measure-inducing. If \(f\) is \(\nu_g\)-integrable, then define the integral operator of \(f\) w.r.t. \(\nu_g\) as the Lebesgue integral given by
\begin{align*}
\psi_g(f):=\int_\Xi f(x)\de \nu_g\,.
\end{align*}
Further, if \(g_I\) exists and is measure-inducing for all \(I\subseteq\{1,\ldots,d\}\,,\) then define the integration by parts operator of \(f\) w.r.t. \(g\) by
\begin{align}\label{defdpio}
\pi_g(f) = \sum_{I\subseteq \{1,\ldots,d\}\atop I\ne \emptyset} 
\int_{\Xi_I} f_I(u)\de \nu_{g_I}(u)
+ f_\emptyset g_\emptyset\,,
\end{align}
whenever all integrals exist. In contrast to \cite{Lux-2017,Tankov-2011}, here the notation for both \(\psi_g(f)\) and \(\pi_g(f)\) is chosen in the way that the measure-inducing function \(g\) appears as subscript and the function \(f\), which refers to the integrands, appears in the round brackets.


In each integral on the right-hand side of \eqref{defdpio} one integrates against the signed measure \(\nu_{g_I}\) induced by the lower \(I\)-marginal of \(g\,.\) Note that, in general, \(\nu_{g_I}\) does not coincide with the \(I\)-marginal measure \(\nu_g^I\) of the signed measure \(\nu_g\,,\)  see Example \ref{exmmamibl}\eqref{exmmamibl1}. Further,  the lower \(I\)-marginal \(g_I\) of a measure-inducing function \(g\) is not necessarily measure-inducing, see Example \ref{exmindgb}. Among others, these issues concern the integration by parts formula in \cite{Lux-2017}.

\subsection{Integration by parts for measure-inducing functions}

In this subsection, we derive conditions on functions \(f,g\colon \Xi\to \R\) such that the integration by parts formula \(\psi_h(f)=\pi_f(\overline{g})\) holds true. This requires the existence of \(f_\emptyset\) and \(g_\emptyset\) in \eqref{defdpio}.

Assuming that the functions \(f\) and \(g\) admit componentwise limits, we say that they have \emph{no common discontinuities on the same side} of each point, if for all inner points \(x\in \Xi\) and for all \(i\in \{1,\ldots,d\}\) the implications 
\begin{align}\label{defcomdiscy}
\begin{split}
\lim_{\varepsilon \downarrow 0} f(x-\varepsilon e_i) \ne f(x)  ~~~ &\Longrightarrow ~~~ \lim_{\varepsilon\downarrow 0} g(x-\varepsilon e_i) = g(x) \\\text{and} ~~~ \lim_{\varepsilon \downarrow 0} f(x+\varepsilon e_i) \ne f(x) ~~~ & \Longrightarrow ~~~ \lim_{\varepsilon\downarrow 0} g(x+\varepsilon e_i) = g(x)
\end{split}
\end{align}
are satisfied, i.e., for any \(i\,,\) if \(f\) has a jump discontinuity on the left/right of \(x\) in the \(i\)th component, then \(g\) is left-/right-continuous at \(x\) in the \(i\)th component. Conversely, whenever \(g\) has a jump discontinuity on the left/right of \(x\) in the \(i\)th component, then \(f\) is left-/right-continuous at \(x\) in the \(i\)th component.
Obviously, a left-continuous and a right-continuous measure-inducing function have no common discontinuity at the same side of each point.

Using the notation introduced above, Theorem \ref{proppif} states for \(f \in \cF_{mi}^{c,l}(\Xi)\) and \(g\in \cF_{mi}^c(\Xi)\,,\) where \(f_\emptyset\) exists, \(g\) is bounded and grounded, and \(f\) and \(g\) have no common discontinuities at the same side of each point, that
\begin{align}\label{genformi}
\psi_g(f) = \pi_f({\overline{g}})\,,
\end{align}
whenever all integrals exist.
%
%
As a direct consequence, we obtain the following result under some continuity assumptions.

\begin{corollary}[Integration by parts for (left-/right-)continuous functions]\label{corpif}~\\
Let \(f \in \cF_{mi}^{c,l}(\Xi)\) and \(g\in \cF_{mi}^c(\Xi)\,.\) Assume that \(f_\emptyset\) exists and that \(g\) is bounded and grounded.
If either
\begin{enumerate}[(i)]
\item \label{corpif1} \(f\) is left-continuous and \(g\) is right-continuous, or
\item \label{corpif2} \(f\) is right-continuous and \(g\) is left-continuous, or
\item \label{corpif3} at least one of the functions \(f\) and \(g\) is continuous,
\end{enumerate}
then
\begin{align}\label{eqcorpif3}
\psi_g(f) = \pi_f({\overline{g}})\,,
\end{align}
whenever all integrals exist.
\end{corollary}

\begin{remark}
\begin{enumerate}[(a)]
\item By Lemma \ref{lemmargdisgrou}\eqref{lemmargdisgrou1}, any grounded and bounded function \(g\in \cF_{mi}^c(\Xi)\) is in \(g\in \cF_{mi}^{c,l}(\Xi)\,.\) Since for \(f,g\in\cF_{mi}^{c,l}(\Xi)\) a uniquely determined componentwise left-continuous or right-continuous version exists, we obtain from Theorem \ref{proppif} that \(\psi_g(f_-) = \pi_f(\overline{g_+})\) and \(\psi_g(f_+) = \pi_f(\overline{g_-})\,.\)
\item By Sklar's Theorem, every \(d\)-dimensional distribution function \(F\) can be decomposed into a composition of its univariate marginal distribution functions \(F_1,\ldots,F_d\) and a \(d\)-variate copula \(C\,,\) i.e., \(F=C\circ (F_1,\ldots,F_d)\,,\) see \cite[Theorem 2.10.9]{Nelsen-2006}. As a special case of Theorem \ref{proppif}, we obtain for the underlying space \(\Xi=\R_+^d\) that
\begin{align}\label{luxst}
\psi_F(f)=\pi_f(\overline{F})
\end{align}
whenever \(f\in \cF_{mi}^{c,l}(\R_+^d)\) is left-continuous and \(F_i(0)=0\) for all \(i\in \{1,\ldots,d\}\,.\) Formula \eqref{luxst} is given in \cite[Proposition 5.3]{Lux-2017}, see also \cite[Theorem 3.7]{Yanez-2020}, where, however, the constant \(f_\emptyset g_\emptyset\) in \eqref{defdpio} is missing and several necessary assumptions on \(f\) and \(F_i\) are not posed. 
For example, if \(f\) in \eqref{luxst} is assumed to be right-continuous, then jump discontinuities have to be considered as Example \ref{exajumdis} shows.
Note that \(F_I=C_I\circ (F_1,\ldots,F_d)\) for all \(I\in \{1,\ldots,d\}\,,\) \(I\ne \emptyset\,,\) since we assume \(F_i(0)=0\,.\)
\end{enumerate}
\end{remark}

The following example illustrates that formula \eqref{genformi} cannot be applied if \(f\) and \(g\) have a common discontinuity at the same side of a point. For \(d=1\,,\) a general integration by parts formula that involves corrective terms due to common discontinuities at the same side of some points is given in \cite[Theorem II.19.3.13]{Hildebrandt-1963} for the so-called Young integral.

\begin{example}[Common discontinuity at the same side of a point]\label{exajumdis}~\\
Consider the functions \(f,g\colon \R_+\to \R\) defined by \(f(x)=g(x)=\1_{\{x\geq 1\}}\,.\) Then \(f\) and \(g\) induce the one-point probability measure \(\delta_{x}\) in \(x=1\) with mass \(1\,,\) i.e., \(\nu_f=\nu_g=\delta_{1}\,.\) Since both \(f\) and \(g\) are right-continuous and have a common discontinuity on the same side at \(x=1\,,\) Theorem \ref{proppif} does not apply and formula \eqref{genformi} is not valid. Indeed, we have
\begin{align*}
\psi_g(f)&= \int_{\R_+} f(x)\de \nu_g(x) = f(1) \cdot \nu_g(\{1\}) = 1\cdot 1 = 1 \,, ~~~ \text{but}\\
\pi_f(\overline g)&= \int_{\R_+} \overline{g}(x) \de \nu_f(x) + \overline{g}_\emptyset f_\emptyset = \int_{\R_+} \1_{\{x<1\}} \de \nu_f(x) = 0 \cdot \nu_f(\{1\}) = 0\cdot 1 =0\,,
\end{align*}
using that \(\overline{g}(x) = g(\infty)-g(x)=1-\1_{\{x\geq 1\}}= \1_{\{x<1\}}\,,\) \(\overline{g}_\emptyset=g(\infty)-g(0)=1\,,\) and \(f_\emptyset = f(0)=0\,.\)
\end{example}

If \(f\) is grounded, the integration by parts operator \(\pi_f(g)\) reduces for every \(\nu_f\)-integrable function \(g\) to
\begin{align}\label{dogro}
\pi_f(g) =\int_{\Xi} g(x)\de \nu_f(x)
\end{align}
because \(f_I(x)=0\) for all \(I\subsetneq \{1,\ldots,d\}\) and \(x\in \Xi\,.\) Hence, if both functions \(f\) and \(g\) are grounded, the integration by parts formula \eqref{genformi} simplifies as follows.

\begin{corollary}[If both functions are grounded]\label{corgrf}~\\
Let \(f,g\in \cF^c_{mi}(\Xi)\) be grounded. Assume that \(h\) is bounded. If \(f\) and \(g\) have no common discontinuities at the same side of each point, then
\begin{align*}
\int_{\Xi} f(x)\de \nu_g(x) = \int_{\Xi} \overline{g}(x)\de \nu_f(x)\,.
\end{align*}
\end{corollary}
A similar formula, where \(f\) and \(g\) are copulas, is given in \cite[Lemma 4.1]{Fuchs-2016}.


\subsection{Transformation formula}

In this subsection, we establish a transformation formula for the integration by parts operator, which corresponds to the change of variables for Lebesgue integrals.

We denote a function \(F\colon \Xi_i\to [0,1]\) that is non-decreasing, right-continuous and satisfies \(\lim_{x\downarrow a_i}F(x)=0\) and \(\lim_{x\uparrow b_i} F(x)=1\) as \emph{univariate distribution function on \(\Xi_i\)}. Note that we do not allow any probability mass at \(a_i=\inf(\Xi_i)\) and \(b_i=\sup(\Xi_i)\,.\) Its \emph{generalized inverse} \(F^{[-1]}\colon [0,1]\to \Xi_i\cup \{a_i,b_i\}\) is defined by
\begin{align}\label{defgeninvy}
F^{[-1]}(t):=\inf\{x\in \Xi_i \colon F(x)\geq t\}
\end{align}
with the convention that \(\inf \emptyset = b_i\,.\)
Note that \(F^{[-1]}\) is left-continuous.
 For every measure-inducing function \(g\colon [0,1]^d\to \R\,,\) for every distribution function \(F_i\colon \Xi_i\to \R\,,\) \(1\leq i \leq d\,,\) and for every measurable function \(f\colon \Xi\to \R\,,\) a well-known identity for the expectation operator states that 
\begin{align*}
\psi_{g\circ(F_1,\ldots,F_d)}(f) = \psi_g\left(f\circ (F_1^{[-1]},\ldots,F_d^{[-1]})\right)
\end{align*}
whenever the integrals exist, compare \cite[Theorem 2]{Winter-1997}.
A similar result holds true for the integration by parts operator as follows.

\begin{proposition}[Transformation formula]\label{lemmargtraf}~\\
Let \(f\in \cF_{mi}^{c,l}(\Xi)\) and, for all \(i\in \{1,\ldots,d\}\,,\)  let \(F_i\) be a univariate distribution function on \(\Xi_i\,.\) Then, for every measurable function \(g\colon [0,1]^d\to \R\,,\) it holds that
\begin{align}\label{eqlemmargtraf}
\pi_{f\circ (F_1^{[-1]},\ldots,F_d^{[-1]})}(g) = \pi_f(g\circ(F_1,\ldots,F_d))
\end{align}
provided all integrals exist.
\end{proposition}

Note that the transformation formula \eqref{eqlemmargtraf} for the integration by parts operator relates to the transformation of signed measures in \eqref{eqpropmeatraf} as follows:
If for all \(i\in I=\{i_1,\ldots,i_k\}\subseteq \{1,\ldots,d\}\,,\) \(I\ne \emptyset\,,\) \(F_i\) is continuous and strictly increasing, then \(F_i^{[-1]} = F_i^{-1}\) for all \(i\in I\) and the signed measure induced by \(f_I\circ (F_{i_1}^{[-1]},\ldots,F_{i_k}^{[-1]})\) is the image of \(\nu_{f_I}\) under \((F_{i_1},\ldots,F_{i_k})\,.\)

\section{Convergence results and extension of the Lebesgue integral to semi-copulas}\label{secvonv}

%
In this section, convergence results are derived for the integrals \(\int f \de \mu_n\) for classes of measure-inducing functions \(f\,,\) where \(\mu_n = \nu_{g_n}\) is the signed measure induced by a grounded and measure-inducing function \(g_n\,,\) which converges pointwise to a function \(g\,.\) 
In general, the 'limit measure' is no longer a signed measure, since it may exhibit infinite positive and negative mass on compact sets. This allows an extension of the Lebesgue integral to the case where one integrates with respect to a quasi-copula, which, in general, does not induce a signed measure.
We restrict us to the domains \(\Xi=\R_+^d\) or \(\Xi=[0,1]^d\) for which, under the continuity assumption \eqref{contboun1}, \(f_\emptyset\) and \(g_\emptyset\) defined by \eqref{eqfmess} exist. Convergence results for general domain \(\Xi\) can be obtained analogously, using the transformation invariance of signed measures in Proposition \ref{propmeatraf} and the transformation invariance of the integration by parts operator in Proposition \ref{lemmargtraf}.

%

\subsection{Compact domain}

In this subsection, we consider the case where the domain \(\Xi\) is compact and, hence, every measure-inducing function on \(\Xi\) induces a finite signed measure by \eqref{defmind}, see Remark \ref{remindmeas}\eqref{remindmeas1}.
as discussed above, we choose without loss of generality \(\Xi=[0,1]^d\,.\)

We first show convergence of the integrals in the integration by parts formula \eqref{genformi1} under the standard assumptions that one of the functions is continuous and the other is measure-inducing. 
Therefore, let \(f\colon [0,1]\to \R\) be continuous and \(g\in \cF_{mi}^c([0,1]^d)\) be grounded. Then \(f\) is bounded and \(g\) induces a finite signed measure (see Remark \ref{remindmeas}\eqref{remindmeas1}) and thus \(\psi_g(f)\) exists. Whenever \(f_n\in \cF_{mi}^{c,l}(\Xi)\,,\) \(n\in \N\,,\) is continuous and converges pointwise to \(f\,,\) one has
\begin{align}\label{conhhh1}
\pi_{f_n}(\overline{g}) = \psi_g(f_n) \to \psi_g(f)
\end{align}
due to Corollary \ref{corpif}\,\eqref{corpif3} and the dominated convergence theorem, i.e., the left-hand side and therefore also the right-hand side of \eqref{genformi1} converges. Similarly, for \(g\) being grounded and continuous, let \(g_n\in \cF_{mi}^c([0,1]^d)\,,\) \(n\in \N\,,\) be a sequence of grounded and continuous functions such that \(g_n\to g\) pointwise. For \(f\in \cF_{mi}^{c,l}([0,1]^d)\,,\) it follows that
\begin{align}\label{conhhh2}
\psi_{g_n}(f) = \pi_f(\overline{g_n}) \to \pi_f(\overline{g})\,,
\end{align}
i.e., the right-hand side and thus also the left-hand side in \eqref{genformi1} converges.
Since a function is in \(\cF_{mi}^{c,l}([0,1]^d)\) if and only if it has bounded Hardy-Krause variation (see Theorem \ref{propbhkv}), we obtain from \eqref{conhhh1} and \eqref{conhhh2} the following result which is a version of integration by parts in \cite[Proposition 2]{Zaremba-1968} for measure-inducing functions. Note that a grounded function is in the class \(\cF_{mi}^{c,l}([0,1]^d)\,,\) whenever it is in \(\cF_{mi}^{c}([0,1]^d)\,,\) see Lemma \ref{lemmargdisgrou}\eqref{lemmargdisgrou1}.

\begin{proposition}[If one of the functions is continuous]\label{propconvfin}~\\
For \(f,g\colon [0,1]^d\to \R\,,\) assume that \(g\) is grounded. If one of the functions is continuous and the other is in \(\cF_{mi}^{c,l}([0,1]^d)\,,\) then the integrals in \eqref{genformi1} converge and
\begin{align}
\psi_g(f)= \pi_f(\overline{g})\,.
\end{align}
\end{proposition}

By the above result, the Lebesgue integral of a measure-inducing function \(f\in \cF_{mi}^{c,l}([0,1]^d)\) can be extended to the case where one integrates with respect to a grounded and continuous function \(g\) which is not assumed to induce a signed measure.

The integration by parts formula \eqref{genformi1} is also valid when both functions exhibit jump discontinuities that do not occur at the same side of each point.
The following theorem gives sufficient conditions on the convergence of the integrals \((\psi_{g_n}(f))_{n\in \N}\) if one of the functions is left-continuous and the other is right-continuous. 


\begin{theorem}[Left-/right-continuous case]\label{theconvfin}~\\
For all \(n\in \N\,,\) let \(g_n\in \cF_{mi}^c([0,1]^d)\) be right-continuous (left-continuous) and grounded. Let \(f\in \cF_{mi}^{c,l}([0,1]^d)\) be left-continuous (right-continuous). Assume that there exists \(M\in \N\) and a function \(g\colon [0,1]^d\to \R\) such that
\begin{enumerate}[(i)]
\item \label{theconvfin1}\(|g_n(x)|\leq M\) for all \(x\in [0,1]^d\) and \(n\in \N\,,\) and
\item \label{theconvfin2} \(g_n(x)\to g(x)\) for all \(x\in [0,1]^d\,.\)
\end{enumerate}
Then, \(\pi_f(\overline{g})\) exists and
\begin{align}\label{genformi18}
\psi_{g_n}(f) \xrightarrow{n\to \infty} \pi_f(\overline{g})\,.
\end{align}
\end{theorem}


\begin{remark}
\begin{enumerate}[(a)]
\item If the function \(g\) in Theorem \ref{theconvfin} is in \(\cF_{mi}^c([0,1]^d)\) and left-continuous/right-continuous, then \(\psi_{g_n}(f) \to \psi_g(f)\) due to \eqref{genformi18} and \eqref{eqcorpif3}, which is a variant of weak convergence of measures.
\item 
Theorem \ref{theconvfin} allow an extension of the Lebesgue integral of a left-continuous function \(f\in \cF_{mi}^{c,l}([0,1]^d)\) to the case where one integrates w.r.t. a continuous semi-copula \(S\,.\) For example, consider the approximation \((S_n)_{n\in \N}\) of \(S\) given by
\begin{align}
\label{semappx} S_n&=S\circ (F_n,\ldots,F_n)\,,\\
\label{semappx2} F_n(x)&= \begin{cases}
0 & \text{if } x<\tfrac 1 n\,,\\
\tfrac k n &\text{if } x \in \left[\tfrac k {n+1}, \tfrac {k+1}{n+1}\right)\,, k=1,\ldots,n\,,\\
1& \text{if } x\geq 1\,.
\end{cases}
\end{align}
Then, \(S_n\) is right-continuous, measure-inducing, grounded, and fulfils continuity conditions \eqref{contboun1} and \eqref{contboun2} for all \(n\in \N\,.\) Hence, we may define the Lebesgue integral of \(f\) w.r.t. the semi-copula \(S\) due to Theorem \ref{theconvfin} by
\begin{align}\label{exteli1}
\int_{[0,1]^d} f(u) \de S(u):=\pi_f(\overline{S}) = \lim_{n\to \infty} \psi_{S_n}(f) = \lim_{n\to \infty}\int_{[0,1]^d} f(u) \de S_n(u)\,,
\end{align}
see \cite[Definition 5.4]{Lux-2017} for the special case of a quasi-copula.
More generally, \eqref{exteli1} can be defined by an application of Proposition \ref{propconvfin} for general \(f\in \cF_{mi}^{c,l}([0,1]^d)\) using a linearly interpolated version of \(S_n\) in \eqref{semappx}.
\end{enumerate}
\end{remark}

\subsection{Non-compact domain}

In the case where the underlying space is non-compact, e.g., \(\Xi=\R^d\,,\) \(\Xi=\R_+^d\,,\) or \(\Xi=[0,1)^d\,,\) the induced signed measures \(\{\nu_{f_I}\}_{I\subseteq \{1,\ldots,d\}}\) are not necessarily finite. So, conditions \eqref{theconvfin1} and \eqref{theconvfin2} of Theorem \ref{theconvfin}, which yield boundedness of \(g\), are not sufficient for the existence of the integrals \(\int_{[0,1]^{|I|}} \overline{g}_I \de \nu_{f_I}\,,\) \(I\subseteq \{1,\ldots,d\}\,.\)
The following theorem provides the convergence of the integrals \((\psi_{g_n}(f))_{n\in \N}\) using a simple integrability condition on \(f\,.\) It extends the integration by parts formula in \cite[Proposition 2]{Tankov-2011} to the multivariate case and to integrators that are approximated by signed measures.

\begin{theorem}[Convergence on non-compact domain]\label{theconvthsec}~\\
For all \(n\in \N\,,\) let \(g_n\colon \R^d\to \R\) be right-continuous, grounded, and measure-inducing.
Let \(f\in \cF_{mi}^{c,l}(\R^d)\) be left-continuous with \(f_\emptyset\in \R\,.\)
Assume that there are distribution functions \(F_1,\ldots,F_d\colon \R\to [0,1]\,,\) a function \(g\colon \R^d\to \R\,,\) and \(\alpha>0\)
such that
\begin{enumerate}[(i)]
\item \label{theconvthsec1} \(g_n(x)\to g(x)\) for all \(x\in \R^d\,,\)
\item \label{theconvthsec1b} \(g_n^\emptyset \to g^\emptyset\,,\)
\item \label{theconvthsec2} \(|\overline{g_n}(x)| \leq \alpha\, (1-\max_{i=1,\ldots,d}\{F_i(x_i)\})\) for all \(x=(x_1,\ldots,x_d)\in \R^d\) and for all \(n\in \N\,,\) and
\item \label{theconvthsec3} the Lebesgue integral \(\int_0^1 f_I\left(F_{i_1}^{[-1]}(u),\ldots,F_{i_k}^{[-1]}(u)\right) \de u\) exists for all \(I=\{i_1,\ldots,i_k\}\subseteq \{1,\ldots,d\}\,,\) \(I\ne \emptyset\,.\)
\end{enumerate}
Then, \(\pi_f(\overline{g})\) is finite and
\begin{align}\label{convtheconvthsec}
\psi_{g_n}(f) \xrightarrow{n\to \infty} \pi_f(\overline{g})\,.
\end{align}
\end{theorem}

\begin{remark}
By Assumption \eqref{theconvthsec2} of Theorem \ref{theconvthsec}, \((\overline{g_n})_{n\in \N}\) is uniformly asymptotically bounded by the survival function \(\overline{F_{X_1,\ldots,X_d}}\) of a so-called comonotonic random vector \((X_1,\ldots,X_d) = (F_1^{[-1]}(U),\ldots,F_d^{[-1]}(U))\,,\) \(U\) uniformly distributed on \((0,1)\,,\) having marginal distribution functions \(F_1,\ldots,F_d\,.\) Indeed, by \eqref{eqtheconvthsec0}, we have
\[\overline{F_{X_1,\ldots,X_d}}(x)=\overline{M^d}(F_1(x_1),\ldots,F_d(x_d))=1-\max\{F_i(x_i)\}\]
for all \(x=(x_1,\ldots,x_d)\in \R^d\,,\) where \(M^d(u_1,\ldots,u_d):= \min\{u_1,\ldots,u_d\}\) denotes the upper Fr\'{e}chet copula.
 For \(x\to \infty\,,\) Assumption \eqref{theconvthsec2} yields, in particular, that \((\overline{g_n})_n\) is uniformly bounded by \(\alpha\,.\)
Assumption \eqref{theconvthsec3} states that for all \(I=\{i_1,\ldots,i_k\}\subseteq\{1,\ldots,d\}\,,\) \(I\ne \emptyset\,,\) the expectation of \(f_I(X_{i_1},\ldots,X_{i_k})\)  exists which guarantees, inductively by the integration by parts formula, the \(\nu_{f_I}\)- integrability of \(\overline{F_{X_{i_1},\ldots,X_{i_k}}}\,.\)
Hence, Assumption \eqref{theconvthsec2} and \eqref{theconvthsec3} yield that \(|(\overline{g_n})_I|\) is dominated by an integrable function for all \(n\), which allows an application of the dominated convergence theorem to obtain the convergence in \eqref{convtheconvthsec}.
\end{remark}

As will be shown below, Theorem \ref{theconvthsec} allows an extension of the Lebesgue integral of a possibly unbounded measure-inducing function \(f\colon [0,1)^d\) to the case where one integrates with respect to an \(L\)-Lipschitz semi-copula. Note that semi-copulas describe the dependence structure of capacities which generalize measures by requiring only monotonicity instead of \(\sigma\)-additivity, see, e.g., \cite{Durante-2016}.  For \(L=1\,,\) i.e., in the case of quasi-copulas, which do not necessarily induce signed measures, several applications to finance are given in \cite{Ansari-2021,Lux-2017,Yanez-2020,Tankov-2011} in the context of improved price bounds.
We make use of the following lemma which states that the survival function of an \(L\)-Lipschitz semi-copula is pointwise bounded by a multiple of the survival function of the upper Fr\'{e}chet copula defined by \eqref{defuppfrecop}.

\begin{lemma}\label{propsvsemcopbou}
For \(L>0\,,\) let \(S\colon [0,1]^d\to [0,1]\) be an \(L\)-Lipschitz semi-copula. Then it holds that
\begin{align*}
|\overline{S}(u)| \leq (2^{d-2}L+1) (1-\max_{1\leq i \leq d}\{u_i\})
\end{align*}
for all \(u=(u_1,\ldots,u_d)\in [0,1]^d\,.\)
\end{lemma}

The following result is a consequence of Theorem \ref{theconvthsec} and Lemma \ref{propsvsemcopbou}. By a simple integrability condition, it provides
the convergence of the integrals w.r.t. discretised \(L\)-Lipschitz semi-copulas in the case of a non-compact domain.

\begin{corollary}[Convergence on non-compact domain]\label{corconsemcopres}~\\
For all \(i\in \{1,\ldots,d\}\) and \(n\in \N\,,\) let \(F_{i,n}\colon \R\to[0,1]\) be the distribution function of a distribution on \((0,1)\) with finite support. Let \(S\) be an \(L\)-Lipschitz semi-copula for some \(L>0\,,\) and let \(f\in \cF_{mi}^{c,l}([0,1)^d)\) be left-continuous. Assume that
\begin{enumerate}[(i)]
\item \label{corconsemcopres1} \(F_{i,n}(u)\to u\) for all \(i\in \{1,\ldots,d\}\) and for all \(u\in [0,1]\,,\)
\item \label{corconsemcopres3} the Lebesgue integral \(\int_0^1 f_I(u,\ldots,u) \de u\) exists for all \(I\subseteq\{1,\ldots,d\}\,,\) \(I\ne \emptyset\,.\)
\end{enumerate}
Then \(\pi_f(\overline{S})\) exists and
\begin{align*}
\psi_{S\circ (F_{1,n},\ldots,F_{d,n})}(f) \to \pi_f(\overline{S})\,.
\end{align*}
\end{corollary}

\begin{remark}
\begin{enumerate}[(a)]
\item Similar to \eqref{exteli1}, Corollary \ref{corconsemcopres} allows an extension of the Lebesgue integral to the case where one integrates w.r.t. a semi-copula, where the integrand now can be an unbounded measure-inducing function. More precisely, let \(f\in \cF_{mi}^{c,l}([0,1)^d)\) be left-continuous and choose, for example, \(F_{i,n}:=F_n\) for all \(i\in \{1,\ldots,d\}\) with \(F_n\) defined by \eqref{semappx}. 
Then, \(F_{i,n}\) satisfies condition \eqref{corconsemcopres1} of Corollary \eqref{corconsemcopres}.
For \(S\in \cS_{d,L}\,,\) \(L>0\,,\) consider the approximation \(S_n:=S\circ (F_{1,n},\ldots,F_{d,n})\,.\) Then the Lebesgue integral of \(f\) w.r.t. \(S\) is defined by
\begin{align}\label{exteli2}
\int_{[0,1)^d} f(u) \de S(u):=\pi_f(\overline{S}) = \lim_{n\to \infty} \psi_{S_n}(f) = \lim_{n\to \infty}\int_{[0,1)^d} f(u) \de S_n(u)
\end{align}
whenever \(f\) satisfies the integrability condition \eqref{corconsemcopres3}. Since quasi-copulas are \(1\)-Lipschitz semi-copula, the integral in \eqref{exteli2} is also defined for quasi-copulas.
\item Let \(F_1,\ldots,F_d\) be distribution functions on \((0,\infty)\,.\) Applying the transformation formula \eqref{eqlemmargtraf}, the Lebesgue integral of a left-continuous function \(f\) w.r.t. the composition \(G:=S\circ(F_1,\ldots,F_d)\,,\) \(S\in \cS_{d,L}\,,\) is defined by
\begin{align}\label{exteli3}
\int_{\R_+^d} f(x) \de G(x) := \pi_f(\overline{G}) = \pi_{f\circ (F_1^{[-1]},\ldots,F_d^{[-1]})}(\overline{S})\,,
\end{align}
whenever \(h:=f\circ (F_1^{[-1]},\ldots,F_d^{[-1]})\in \cF_{mi}^{c,l}\) satisfies the integrability condition \eqref{corconsemcopres3} of Corollary \ref{corconsemcopres}.
In the special case where \(F\) is a \(d\)-dimensional distribution function on \(\R_+^d\) with marginal distribution functions \(F_1,\ldots,F_d\,,\) there exists by Sklar's theorem a copula \(C\in \cC_d\) such that \(F=C\circ (F_1,\ldots,F_d)\,.\) Then \eqref{exteli3} reduces to the integration by parts formula
\(\psi_F(f) = \pi_f(\overline{F})\) given by \eqref{luxst}.
\end{enumerate}

\end{remark}

\section{Applications to stochastic orderings}\label{appstoord}

As an application of the integration by parts formulas and the convergence results established in Sections \ref{secmr} and \ref{secvonv}, we extend several integral stochastic orderings to the class of (Lipschitz-) continuous semi-copulas. 

For \(d\)-variate distribution functions \(F,F'\) on \(\R^d\,,\) the \emph{integral stochastic ordering} \(\leq_{\cF}\) is defined w.r.t. a class \(\cF\) of measurable functions \(f\colon \R^d\to \R\) by
\begin{align}\label{defstochord}
F\leq_{\cF} F' ~~~ :\,\Longleftrightarrow ~~~ \psi_F(f)\leq \psi_{F'}(f) ~~~ \text{for all } f\in \cF \text{ such that the expectations exist},
\end{align}
see, e.g., \cite{Mueller-Stoyan-2002}.
Important classes of functions that generate stochastic orderings by \eqref{defstochord} are the class \(\cF^{\Delta}_-\) of \(\Delta\)-antitone functions, the class \(\cF^{\Delta}\) of \(\Delta\)-monotone functions, the class \(\cF_{sm}\) of supermodular functions,
the class \(\cF_{dcx}\) of directionally convex functions, and
the class \(\cF_{\uparrow}\) of increasing functions. We refer to Definition \ref{defdelant} for the respective definitions. 

Several integral stochastic orderings have a smooth generator \(\cG\,,\) i.e., \(\leq_{\cG}\) and \(\leq_{\cF}\) are identical, where \(\cG\) only contains smooth functions. It is known that, if the integral ordering \(\leq_{\cF}\) is closed under mixtures and closed under weak convergence, then it has a smooth generator, see, e.g., \cite[Corollary 2.5.6]{Mueller-Stoyan-2002}.
Since any smooth function on \([a,b]\subset \R^d\) is in the class \(\cF_{mi}^{c,l}([a,b])\cap \cF_{mi}^{c,u}([a,b])\,,\) see Corollary \ref{cordom}, integral orderings with a smooth generator can be extended to semi-distribution functions defined as follows. As before, we consider the domain \(\R_+^d\,.\)

\begin{definition}[Semi-distribution function on \(\R_+^d\)]~\\
We denote a function \(H\colon\R_+^d\to [0,1]\) as a \emph{semi-distribution function} if it can be decomposed into a semi-copula \(S\) and univariate distribution functions \(F_1,\ldots,F_d\) with \(F_i(0)=0\) for all \(i\in \{1,\ldots,d\}\) such that
\begin{align}\label{dessdf}
H(x)=S(F_1(x_1),\ldots,F_d(x_d)) ~~~ \text{for all } x=(x_1,\ldots,x_d)\in \R^d\,.
\end{align}
A semi-distribution function \(H\) is called a \emph{quasi-distribution function} if the semi-copula \(S\) in \eqref{dessdf} is a quasi-copula.
\end{definition}

Denote by \(\cH_+^d\) the class of semi-distribution functions on \(\R_+^d\) and denote by \(\cH^d_c\) and \(\cH^d_L\) the subclass of semi-distribution functions for which the semi-copula in \eqref{dessdf} is continuous and \(L\)-Lipschitz-continuous, respectively. Recall that, by Sklar's theorem, \(H\) in \eqref{dessdf} is a distribution function whenever \(S\) is a copula.

We extend the lower orthant and the upper orthant ordering for \(d\)-variate distribution functions to the class \(\cH_+^d\) of semi-distribution functions as follows. 

\begin{definition}[Lower orthant order and upper orthant order on \(\cH^d_+\)]~\\
For \(H,H'\in \cH_+^d\,,\) define the
\begin{enumerate}[(i)]
\item \emph{lower orthant order} \(H\leq_{lo} H'\) by \(H(x)\leq H'(x)\) for all \(x\in \R^d\,,\)
\item \emph{upper orthant order} \(H\leq_{uo} H'\) by \(\overline{H}(x)\leq \overline{H'}(x)\) for all \(x\in \R^d\,.\)
\end{enumerate}
\end{definition}

The following result characterizes the orthant orders on the class \(\cH^d_+\) of semi-distribution functions on \(\R_+^d\,.\)

\begin{theorem}[Characterization of the orthant orders on \(\cH^d\)]\label{cgoo}~\\
For \(H,H'\in \cH^d_{L}\,,\) the following statements hold true:
\begin{enumerate}[(i)]
\item \label{cgoo1} \(H\leq_{lo} H'\) if and only if \(\pi_f(\overline{H})\leq \pi_f(\overline{H'})\) for all \(\Delta\)-antitone functions \(f\colon\R_+^d\to \R\,.\)
\item \label{cgoo2} \(H\leq_{uo} H'\) if and only if \(\pi_f(\overline{H})\leq \pi_f(\overline{H'})\) for all \(\Delta\)-monotone functions \(f\colon\R_+^d\to \R\,.\)
\end{enumerate}
\end{theorem}

For the subclasses of distribution functions and quasi-distribution functions, the above theorem is given in \cite{Rueschendorf-1980} and \cite{Lux-2017}, respectively. We refer to \cite{Ansari-2021,Lux-2017,Tankov-2011} for various applications to improved option price bounds under market-implied dependence information.

\section*{Acknowledgements}

The author thanks Ludger R\"{u}schendorf for a discussion of the subject and Thibaut Lux for a discussion of \cite[Formula (5.3)]{Lux-2017}.
The author gratefully acknowledges the support of the Austrian Science Fund (FWF) project 
{P 36155-N} \emph{ReDim: Quantifying Dependence via Dimension Reduction}
and the support of the WISS 2025 project 'IDA-lab Salzburg' (20204-WISS/225/197-2019 and 20102-F1901166-KZP).


\section*{Appendix A: Properties of measure-inducing functions}\label{Appmif}

\setcounter{subsection}{0}
\setcounter{theorem}{0}
\renewcommand{\thesection}{A.\arabic{section}}
\renewcommand{\thesubsection}{A.\arabic{subsection}}
\renewcommand{\thelemma}{A.\arabic{lemma}}
\renewcommand{\thetheorem}{A.\arabic{theorem}}
\renewcommand{\theproposition}{A.\arabic{proposition}}
\renewcommand{\thedefinition}{A.\arabic{definition}}

In the sequel, we provide various properties of measure-inducing functions concerning left- and right-continuous versions, transformations, marginal functions, the Hardy-Krause variation, the Jordan decomposition, and survival functions. All proofs are deferred to Appendix B.

\subsection{Left- and right-continuity}

Denote a multivariate function as left-/right-continuous if it is componentwise left-/right-continuous. By the following result, a left-/right-continuous measure-inducing function has the desirable property that the induced signed measure is generated by half-open intervals and can be expressed by the difference operator \(\Delta_{x,y}\) in \eqref{defdvdo}.

\begin{proposition}\label{lemfmi}
For \(f\in \cF_{mi}(\Xi)\,,\) the following statements hold true.
\begin{enumerate}[(i)]
\item \label{lemfmi1} If \(f\) is left-continuous, then the induced signed measure \(\nu_f\) is generated by
\begin{align}
\nu_f\left([x,y)\right) = \Delta_{x,y}[f]
\end{align}
for all \(x,y\in \Xi\) with \(x\leq y\,.\)
\item \label{lemfmi2} If \(f\) is right-continuous, then the induced signed measure \(\nu_f\) is generated by
\begin{align}\label{eqfmi2}
\nu_f\left((x,y]\right) = \Delta_{x,y}[f]
\end{align}
for all \(x,y\in \Xi\) with \(x\leq y\,.\)
\item \label{lemfmi3} If \(f\) is continuous, then the induced signed measure \(\nu_f\) is generated by
\begin{align*}
\nu_f\left(B\right) = \Delta_{x,y}[f]
\end{align*}
for all  \(x=(x_1,\ldots,x_d),y=(y_1,\ldots,y_d)\in \R^d\) such that \(a\leq x\leq y\leq b\,,\) where \(B=\bigtimes_{i=1}^d B_i\) with \(B_i\in\{[x_i,y_i],(x_i,y_i],[x_i,y_i),(x_i,y_i)\}\,,\) \(1\leq i\leq d\,.\)
\end{enumerate}
\end{proposition}


For a function \(f\colon \Xi\to\R\,,\) denote by \(f_-\) and \(f_+\) its componentwise left- and right-continuous version, respectively, whenever it exists.
The following lemma states that, under the continuity assumption on \(f\) at the boundary of the domain, the signed measure induced by \(f\) coincides with the signed measure induced by \(f_-\) and \(f_+\,,\) respectively. This property is well-known when \(F\) is a univariate distribution function. Then the corresponding probability measure \(P\) coincides with the measure induced by the left-continuous version \(F_-\,,\) see \eqref{deffP}.

\begin{lemma}\label{lemunimifun}
Let \(f\colon \Xi\to \R\) be a function that satisfies the boundary conditions \eqref{contboun1} and \eqref{contboun2}.
Assume that \(f_-\) and \(f_+\) exist.
If one of the functions \(f\,,\) \(f_-\,,\) and \(f_+\) is measure-inducing, then also the others are measure-inducing and it holds that
\begin{align*}
\nu_f=\nu_{f_-}=\nu_{f_+}\,,
\end{align*}
i.e., the measures induced by \(f\,,\) \(f_-\,,\) and \(f_+\,,\) respectively, coincide.
\end{lemma}

The following example shows that in the above lemma the continuity assumption at the boundary of \(\Xi\) cannot be omitted.

\begin{example}\label{exlemunimifun}
For \(\Xi=[0,1]\,,\) let \(f\colon \Xi\to \R\) be defined by \(f(x)=\1_{\{x>0\}}\,.\) Then \(f\) does not satisfy condition \eqref{contboun1}, i.e., \(f\) is not continuous at the lower boundary of \(\Xi\,.\) Denote by \(\delta_{\{x\}}\) the one-point probability measure with mass \(1\) in \(x\,.\) Then \(\nu_f = \delta_{\{0\}}\,.\) But \(f_+(x)=1\) for all \(x\in [0,1]\) and, thus, \(\nu_{f_+}\) is the null measure. However, if the domain of \(f\) is extended to \((-\varepsilon,1]\) for some \(\varepsilon>0\) and if \(g(x)=\1_{\{x>0\}}\) for \(x\in (-\varepsilon,1]\,,\) then the induced signed measure is the one-point probability measure in \(0\,,\) which coincides with the measure induced by the right-continuous version of \(g\,,\) i.e., \(\nu_g=\nu_{g_+}=\delta_{\{0\}}\,,\) compare Remark \ref{remindmeas}\eqref{eqremindmeas3}.
\end{example}

\subsection{Transformation}

For a signed measure \(\nu\) on \(\Xi\) and for measurable functions \(h_i\colon \Xi_i\to \R\,,\) \(1\leq i \leq d\,,\) denote by \(\nu^{h}:=(\nu^+)^{h}-(\nu^-)^{h}\) the image of \(\nu\) under \(h=(h_1,\ldots,h_d)\,,\) where for a (non-negative) measure \(\mu\) on \(\Xi\,,\) the image \(\mu^h\) is defined by \(\mu^h(A):=\mu(h^{-1}(A))\) for all Borel-measurable sets \(A\subseteq \R^d\,,\) where \(h^{-1}(A)\) denotes the preimage of \(A\) under \(h\,.\)

By the following result, we can restrict our attention of studying measure-inducing functions to domains \(\Xi_i\) of the type \([0,1]\,,\) \([0,1)\,,\) \((0,1]\,,\) or \((0,1)\) because every closed, half-open or open interval can be identified with one of these intervals under a strictly increasing transformation.

\begin{proposition}[Transformation of domains]\label{propmeatraf}~\\
For \(1\leq i \leq d\,,\) let \(\phi_i\colon \Xi_i\to \R\) be continuous and strictly increasing. Denote by \(\phi^{-1}_i\colon \phi_i(\Xi_i) \to \Xi_i\) the inverse of \(\phi_i\,.\) If \(f\colon \Xi \to \R\) is measure-inducing, then \(f\circ (\phi_1^{-1},\ldots,\phi_d^{-1})\) is measure-inducing with induced signed measure
\begin{align}\label{eqpropmeatraf}
\nu_{f\circ (\phi_1^{-1},\ldots,\phi_d^{-1})} = \nu_f^{(\phi_1,\ldots,\phi_d)}\,.
\end{align}
\end{proposition}

\subsection{Marginal functions and marginal measures}

Similar to the definition of the lower \(I\)-marginal in \eqref{defqmargl}, the \emph{upper \(I\)-marginal} \(f^I\) of \(f\) is defined by the function
\begin{align}\label{defqmarg}
f^I(x_{i_1},\ldots,x_{i_k}) := \lim_{x_j\uparrow b_j\atop j\notin I} f(x_1,\ldots,x_d)\,, ~~~x_{i_j}\in \Xi_{i_j} ~\text{for} ~j=1,\ldots,k\,,
\end{align}
whenever the iterated limit exists independent of the arrangement of the limits.
For \(I=\emptyset\,,\) define the constants
\begin{align}\label{eqfmess}
f_\emptyset &:= \lim_{x_j\downarrow a_j\atop j=1,\ldots,d} f(x_1,\ldots,x_d)\,,&
f^\emptyset &:= \lim_{x_j\uparrow b_j\atop j=1,\ldots,d} f(x_1,\ldots,x_d)\,,
\end{align}
whenever the limits exist.

For a measure-inducing function, its \(I\)-marginals are in general not measure-inducing how the following examples show.
\begin{example}\label{exmindgb}
\begin{enumerate}[(a)]
\item \label{exmindgba} Consider \(f\colon \R^2\to \R\) defined by \(f(x,y)=\1_{y\in \Q}(x,y)\,.\) Then \(f\) induces the null measure because \(\Delta_{\varepsilon_1}^1 \Delta_{\varepsilon_2}^2 f(x,y)=0\) for all \(x,y\in \R\,.\) However, the componentwise left- and right-hand limits of \(f\) do not exist and \(f^{\{2\}}(y) = \1_{y\in \Q}(y)\) does not induce a signed measure.
\item Let \(f\colon \R_+^2\to \R\) be defined by \(f(x,y):=xy+\sin(x)\,.\) Then, \(f\) induces the Lebesgue measure on \(\R^2\) by \eqref{defmind}. But \(f_{\{1\}}(x)=\sin(x)\) does not induce a signed measure on \(\R_+\) by \eqref{defmind} because it is bounded and has unbounded total variation.
\end{enumerate}
\end{example}

For \(I=\{i_1,\ldots,i_k\}\subseteq \{1,\ldots,d\}\) and for a signed measure \(\nu\) on \(\Xi\,,\) denote as usual by \(\nu^I\) the \(I\)-marginal signed measure of \(\nu\) defined by
\begin{align}\label{defImargmeas}
\nu^I(A_{i_1}\times \cdots \times A_{i_k}):=\nu(B_1\times \cdots \times B_d)
\end{align}
for all Borel measurable sets \(A_{i_j}\subseteq \Xi_{i_j}\,,\) \(j=1,\ldots,k\,,\) where \(B_i:=A_i\) for \(i\in I\) and \(B_i:=\Xi_i\) for \(i\notin I\,.\)

The following example shows that, in general, the signed measure  induced by the lower/upper \(I\)-marginal of a measure-inducing function does not coincide with the \(I\)-marginal of the induced signed measure.

\begin{example}[\(\nu_{f_I}\ne \nu_f^I\) and \(\nu_{h^I}\ne \nu_h^I\)]\label{exmmamibl}~
\begin{enumerate}[(a)]
\item \label{exmmamibl1} Consider the function \(f\colon \R_+^2\to \R\) defined by \(f(x_1,x_2):=\1_{\{x_1>1\}}\cdot\1_{\{x_2>1\}}\,.\) Then \(f\) induces the measure \(\nu_g=\delta_{(1,1)}\,,\) where \(\delta_x\) denotes the one-point measure with mass \(1\) in \(x\,.\) For \(I=\{1\}\subseteq \{1,2\}\,,\) it holds that \(f_{\{1\}}(x_1)=f(x_1,0)=0\) for all \(x_1\in \R_+\,.\) So, \(f_{\{1\}}\) induces the null measure \(\nu_{f_{\{1\}}}=0\,.\) But the \(\{1\}\)-marginal measure of \(\nu_f\) is the one-point probability measure in \(x_1=1\,,\) i.e., \(\nu_f^{\{1\}}=\delta_1\,.\)
\item Consider the function \(h\colon \R_+^2\to \R\) defined by \(h(x_1,x_2):=\1_{\{x_1<1\}}\cdot\1_{\{x_2<1\}}\,.\) Then \(h\) induces the measure \(\nu_h=\delta_{(1,1)}\,,\) For \(I=\{1\}\subseteq \{1,2\}\,,\) it holds that \(h^{\{1\}}(x_1)=h(x_1,\infty)=0\) for all \(x_1\in \R_+\,.\) So, \(h^{\{1\}}\) induces the null measure \(\nu_{h^{\{1\}}}=0\,.\) But the \(\{1\}\)-marginal measure of \(\nu_h\) is the one-point probability measure in \(x_1=1\,,\) i.e., \(\nu_h^{\{1\}}=\delta_1\,.\)
\end{enumerate}
\end{example}

For the integration by parts formula \eqref{genformi1}, we need that the lower and upper \(I\)-marginals of a measure-inducing function are measure-inducing. Similar to the definition of the class \(\cF_{mi}^{c,l}(\Xi)\) in \eqref{defclassmif}, define the class
\begin{align}
\nonumber \cF_{mi}^{c,u}(\Xi)&:=\left\{ f\in \cF_{mi}^c\mid f^I\in \cF_{mi}^c(\Xi_I) \text{ for all } I\subset \{1,\ldots,d\}\,, I\ne \emptyset\right\}\,.
\end{align}
of measure-inducing functions for which also the upper \(I\)-marginals are measure-inducing.

By the following lemma, the upper \(I\)-marginal of a bounded, grounded, and measure-inducing function can be associated with the \(I\)-marginal of the induced signed measure. This property is used in the proof of formula \eqref{genformi1}.

\begin{lemma}\label{lemmargdisgrou}
Let \(f\in \cF_{mi}^c(\Xi)\) be bounded and grounded. Then,
\begin{enumerate}[(i)]
\item \label{lemmargdisgrou1}\(f\in \cF_{mi}^{c,l}(\Xi)\cap \cF_{mi}^{c,u}(\Xi)\,,\)
\item \label{lemmargdisgrou2} for \(I\subseteq \{1,\ldots,d\}\,,\) it holds that \(\nu_{f^I}=\nu_f^I\,.\)
\end{enumerate}
\end{lemma}

We make use of grounded functions for a characterization of measure-inducing functions through \(\Delta\)-monotone functions, see Theorem \ref{thecfmicl}. In this respect, define for a function \(f\colon \Xi\to \R\) the function \(F^f\colon \Xi\to \R\) by
\begin{align}\label{defFf}
F^f(x)=f(x)-\sum_{k=0}^{d-1}\sum_{I\subsetneq \{1,\ldots,d\}\atop I=\{i_1,\ldots,i_k\}} (-1)^{d-k-1} f_I(x_{i_1},\ldots,x_{i_k})\,,~~~ x=(x_1,\ldots,x_d)\in \Xi\,,
\end{align}
whenever the lower \(I\)-marginals \(f_I\) exist for all \(I\subsetneq \{1,\ldots,d\}\,.\) Note that neither \(f\) nor its lower marginals are assumed to be measure-inducing.

\begin{lemma}[\(F^f\) is grounded]\label{lemFfgr}~\\
Assume that \(f\colon \Xi\to \R\) satisfies the continuity condition \eqref{contboun1}. Then the function \(F^f\) defined by \eqref{defFf} is grounded.
\end{lemma}

\subsection{Characterization of measure-inducing functions}

As one expects, measure-inducing functions and the concept of bounded variation in the sense of the Hardy and Krause are closely related because the integration by parts formula \eqref{genformi1} for measure-inducing functions is related to integration by parts for Stieltjes integrals and because the convergence of the Stieltjes integrals on the compact domain \([0,1]^d\) can be characterized through bounded Hardy-Krause variation, see \cite[Proposition 2]{Zaremba-1968}). Since the domain \(\Xi\) of the functions \(f\) and \(g\) in formula \eqref{genformi1} is not necessarily compact, the functions \(f\) and \(g\) may also have infinite Hardy-Krause variation. In the following, we characterize the classes \(\cF_{mi}^{c,l}(\Xi)\) of measure-inducing functions with compact domain \(\Xi=[0,1]^d\) and with non-compact domain \(\Xi=\R_+^d\) using the concept of Hardy-Krause variation. A characterization for general \(\Xi\) follows similarly with the transformation in Proposition \ref{propmeatraf}.

For the definition of the Hardy-Krause variation, we consider the domain \(\Xi=[a,b]\,.\)
Then, for \(j=1,\ldots,d\,,\) let
\begin{align*}
a_j= x_0^{(j)} < x_1^{(j)} < \cdots < x_{m_s}^{(j)} = b_j
\end{align*}
be a partition of \(\Xi\,.\) For \(I=\{i_1,\ldots,i_k\}\subseteq\{1,\ldots,d\}\,,\) \(I\ne \emptyset\,,\) denote by \(\cP^I\) the partition of \(\Xi_I\) given by
\begin{align}\label{defPk}
\cP^I:=\left\{ [x_{l_1}^{(i_1)}, x_{l_1+1}^{(i_1)} ] \times \cdots \times [x_{l_k}^{(i_k)}, x_{l_k+1}^{(i_k)} ] \,, l_s=0,\ldots,m_s-1 \,, s=1,\ldots,d\right\}\,.
\end{align}
Then for a function \(f\colon \Xi\to \R\,,\) the \emph{Vitali variation} is defined by
\begin{align}\label{defVitvar}
\V(f):= \sup_{\cP^{\{1,\ldots,d\}}} \sum_{[x,y]\in \cP^{\{1,\ldots,d\}}} \left| \Delta_{x,y} [f]\right| \,,
\end{align}
where \(\Delta_{x,y}\) is the difference operator in \eqref{defdvdo}. The \emph{Hardy-Krause variation} anchored at \(a\) is defined by
\begin{align}\label{defHKvar}
\Var(f):=\sum_{k=1}^d \sum_{I \subseteq \{1,\ldots,d\}  \atop {|I|}=k} \sup_{\cP^I} \sum_{[x,y]\in \cP^I} \left| \Delta_{x,y} [f_I]\right| 
\end{align}
where the supremum is taken over all partitions \(\cP^I\) of \(\Xi_I\) into axis-parallel boxes of the type in \eqref{defPk}. Hence, the Hardy-Krause variation anchored at \(a\) is the sum of the Vitali variations of all lower \(I\)-marginal \(f_I\) of \(f\,,\) \(I\subseteq \{1,\ldots,d\}\,,\) \(I\ne \emptyset\,.\) In the literature, the domain \(\Xi\) is often chosen as the unit cube \([0,1]^d\) and the Hardy-Krause variation is usually defined w.r.t. the upper \(I\)-marginals, i.e., anchored at \(b=(1,\ldots,1)\,.\)
For an overview of properties of functions with bounded Hardy-Krause variation, see, e.g., \cite{Basu-2016,Owen-2005}. \\

The following result characterizes the class \(\cF_{mi}^{c,l}([0,1]^d)\) of measure-inducing functions on \([0,1]^d\) for which also the lower or upper \(I\)-marginal functions are measure-inducing. Note that the induced signed measures are all finite because \([0,1]^d\) is compact, see Remark \ref{remindmeas}\eqref{remindmeas1}.

\begin{theorem}[Characterization of {\(\cF_{mi}^{c,l}([0,1]^d)\)}]\label{propbhkv}~\\
Assume that \(f\colon [0,1]^d\to \R\) is right-continuous and fulfils  the continuity conditions \eqref{contboun1} and \eqref{contboun2}. Then the following statements are equivalent:
\begin{enumerate}[(i)]
\item \label{propbhkv1}\(f\) has bounded Hardy-Krause variation,
\item \label{propbhkv2}\(f\in \cF_{mi}^{c,l}([0,1]^d)\,,\)
\item \label{propbhkv3}\(f\in \cF_{mi}^{c,u}([0,1]^d)\,.\)
\end{enumerate}
\end{theorem}

In the sequel, we analyse the class of measure-inducing functions on \(\R_+^d\,.\)
Since the signed measure induced by \(f\in \cF_{mi}^{c,l}(\R_+^d)\) is not necessarily finite and \(f\) may have infinite Hardy-Krause variation, a characterization of this class is more challenging.

As we will see in Theorem \ref{thecfmicl}, every right-continuous function in \(\cF_{mi}^{c,l}(\R_+^d)\) can be constructed by a linear combination of \(\Delta\)-monotone functions which are defined as follows. Recall that \(\Delta_\varepsilon^i\) is the difference operator defined by \eqref{defdopfh}.

\begin{definition}[\(d\)-increasing, \(\Delta\)-monotone, \(\Delta\)-antitone function]\label{defdelant}~\\
Let \(f\colon \Xi\to \R\) be a function. Then, \(f\) is said to be
\begin{enumerate}[(i)]
\item \emph{\(k\)-increasing} for \(k\in \{1,\ldots,d\}\,,\) if for all \(I=(i_1,\ldots,i_k)\subseteq \{1,\ldots,d\}\,,\)
\begin{align*}
\Delta_{\varepsilon_1}^{i_1}\cdots \Delta_{\varepsilon_k}^{i_k} f(x)\geq 0
\end{align*}
for all \(x=(x_1,\ldots,x_d)\in \Xi\) and \(\varepsilon_1,\ldots,\varepsilon_k>0\) such that \(x_{i_j}+\varepsilon_j\in \Xi_j\) for all \(j\in \{1,\ldots,k\}\,,\)
\item \emph{\(k\)-decreasing} if \(-f\) is \(k\)-increasing,
\item \emph{\(\Delta\)-monotone} if \(f\) is \(k\)-increasing for all \(k\in \{1,\ldots,d\}\,,\)
\item \emph{\(\Delta\)-antitone} if \((-1)^k f\) is \(k\)-increasing for all \(k\in \{1,\ldots,d\}\,,\)
\item \emph{supermodular} if \(\Delta_{\varepsilon}^i \Delta_{\varepsilon'}^jf(x) \geq 0\) for all \(x\in \Xi\,,\) \(\varepsilon,\varepsilon'>0\) and \(1\leq i < j \leq d\,,\)
\item \emph{directionally convex} if \(\Delta_{\varepsilon}^i \Delta_{\varepsilon'}^jf(x) \geq 0\) for all \(x\in \Xi\,,\) \(\varepsilon,\varepsilon'>0\) and \(1\leq i \leq j \leq d\,.\)
\end{enumerate}
\end{definition}

Note that, by definition, a \(\Delta\)-monotone/-antitone function is componentwise increasing/decreasing and, thus, its componentwise left- and right-hand limits exist.  Further, any \(\Delta\)-antitone, any \(\Delta\)-monotone, and any directionally convex function is supermodular.\\

%
%


The following result characterizes measure-inducing functions on \(\R_+^d\) for which the marginal functions are not necessarily measure-inducing. Since at least one of the positive and negative part of a signed measure is finite, at least one of the \(\Delta\)-monotone and hence non-decreasing functions \(g\) and \(h\) in \eqref{eqpropcharmif2} has to be bounded. The transformation \(F^f\) is defined by \eqref{defFf} and \(\Var\) denotes the Hardy-Krause variation defined by \eqref{defHKvar}.


\begin{proposition}[Characterization of measure-inducing functions on \(\R_+^d\)]\label{propcharmif}~\\
For a right-continuous function \(f\colon \R_+^d\to \R\,,\) 
the following statements are equivalent:
\begin{enumerate}[(i)]
\item \label{propcharmif1} \(f\) induces a signed measure by \eqref{defmind}.
\item \label{propcharmif2} \(F_f\) generates a signed measure \(\nu\) on \(\R_+^d\) by \(\nu([0,x])=F_f(x)\) for all \(x\in \R_+^d\,.\)
\item \label{propcharmif3} There exist grounded, right-continuous, and \(\Delta\)-monotone functions \(g,h\colon \R_+^d \to \R\,,\) at least one of which is bounded, such that
\begin{align}
\label{eqpropcharmif1} g(0) &=h(0)=0\,,&&\\
\label{eqpropcharmif2} F^f(x)&=g(x)-h(x)~~~&&\text{for all } x \in \R_+^d\,,\\
\label{eqpropcharmif3} \Var(F^f|_{[0,x]})&=\Var(g|_{[0,x]})+\Var(h|_{[0,x]})<\infty~&&\text{for all } x\in \R_+^d\,,\\
\label{eqpropcharmif} \Delta_{x,y}[f]&= \Delta_{x,y}[g]-\Delta_{x,y}[h]~~~&&\text{for all }x,y\in \R_+^d~\text{with } x\leq y\,.
\end{align}
\item \label{propcharmif5} There exist right-continuous and \(d\)-increasing functions \(g,h\colon \R_+^d\to \R\,,\) at least one of which is bounded, such that \eqref{eqpropcharmif} is fulfilled.
\end{enumerate}
Further, the functions \(g\) and \(h\) in \eqref{propcharmif3} are uniquely determined.
\end{proposition}

In the integration by parts formula \eqref{genformi1}, the function \(f\) is assumed to be from the class \(\cF_{mi}^{c,l}(\Xi)\).
The following result shows that each such \(f\) is essentially a linear combination of \(\Delta\)-monotone functions.  Similar to Proposition \ref{propcharmif}, for all \(I\subseteq \{1,\ldots,d\},\) \(I\ne \emptyset,\) at least one of the functions \(g_I\) and \(h_I\) in \eqref{eqthecfmicl22} is bounded.

\begin{theorem}[Characterization of the class \(\cF_{mi}^{c,l}(\R_+^d)\)]\label{thecfmicl}~\\
For a  function \(f\colon \R_+^d \to \R\,,\) the following statements are equivalent:
\begin{enumerate}[(i)]
\item \label{thecfmicl1} \(f\in \cF_{mi}^{c,l}(\R_+^d)\) is right-continuous,
\item \label{thecfmicl2} For all \(I\subseteq\{1,\ldots,d\}\,,\) \(I\ne \emptyset\,,\) there exist  right-continuous, \(\Delta\)-monotone, and grounded functions \(g_{(I)},h_{(I)}\colon \R_+^{|I|}\to \R\,,\) at least one of which is bounded, such that 
\begin{align}
g_{(I)}(0)&=h_{(I)}(0)=0\,,\\
\label{eqthecfmicl22} F^{f_{(I)}}(y)&= g_{(I)}(y)-h_{(I)}(y)\,, ~~~\text{for all } y\in \R_+^{|I|}\,,\\
\Var(F^{f_{(I)}}|_{[0,y]})&=\Var({g_{(I)}}|_{[0,y]})+\Var({h_{(I)}}|_{[0,y]})<\infty~~~\text{for all } y\in \R_+^{|I|}\,,\\
\label{eqthecfmicl2}f(x) &= f(0)+ \sum_{k=1}^d \sum_{I\subseteq \{1,\ldots,d\}\atop I=\{i_1,\ldots,i_k\}} \left[g_{(I)}(x_{i_1},\ldots,x_{i_k}) - h_{(I)}(x_{i_1},\ldots,x_{i_k})\right] 
\end{align}
for all \(x=(x_1,\ldots,x_d)\in \R_+^d\,.\)
\end{enumerate} 
The functions \(g_{(I)}\) and \(h_{(I)}\,,\) \(I\subseteq\{1,\ldots,d\}\,,\) \(I\ne \emptyset\,,\) in \eqref{thecfmicl2} are uniquely determined.
\end{theorem}

\begin{remark}
\begin{enumerate}[(a)]
\item If \(f\colon \R_+^k\to \R\) is grounded, \(k\in \N\,,\) then \(f\) is \(\Delta\)-monotone if and only if it is \(k\)-increasing. Hence, for all \(I\subseteq \{1,\ldots,d\}\,,\) the uniquely determined right-continuous functions \(g_{(I)}\) and \(h_{(I)}\) in Theorem \ref{thecfmicl}\eqref{thecfmicl2} are measure-generating functions, i.e., there exist measures \(\mu_I\) and \(\eta_I\) on \(\cB(\R^{|I|}_+)\) such that
\begin{align}\label{defmgfun}
\mu_I([0,x])&=g_{(I)}(x)\,, && \eta_I([0,x])=h_{(I)}(x)\,,
\end{align}
\(x\in \R_+^{|I|}\,.\)
\item 
Since \(\Delta\)-monotone functions are increasing, their componentwise right- and left-hand limits exist. Hence, dispensing with right-continuity, Theorem \ref{thecfmicl} still holds true.
\end{enumerate}
\end{remark}

\subsection{Survival functions}


As the following example shows, care must be taken in the integration by parts formula \eqref{genformi1}, where the lower \(I\)-marginal of the survival function of \(g\) has to be determined and not the survival function of the lower \(I\)-marginal.


\begin{example}[\(\overline{f^I} \ne \overline{f}^I\,,\) \(\overline{f_I}\ne \overline{f}_I\)]\label{exuimsf}
Let \(f\colon [0,1]^2\to [0,1]\) be a bivariate copula, i.e., \(f\) is a bivariate distribution function on \([0,1]^2\) that is grounded and has uniform univariate marginals, so \(f(u,0)=0\,,\) \(f(0,v)=0\,,\) \(f(u,1)=u\) and \(f(1,v)=v\) for all \(u,v\in [0,1]\,.\) It holds that \(\overline{f}(u,v)=f(u,v)-u-v+1\,.\)
\begin{enumerate}[(a)]
\item  For \(I=\{1\}\) it follows that \(f^{\{1\}}(u)=f(u,1)=u\) using the uniform marginal property. This yields \(\overline{f^{\{1\}}}(u)=1-u\ne 0 =\overline{f}(u,1)= \overline{f}^{\{1\}}(u)\) for \(u\in [0,1)\,.\)
\item For \(I=\{1\}\,,\) it follows that \(f_{\{1\}}(u)=f(u,0)=0\) using that \(f\) is grounded. This implies for \(u\in [0,1)\) that \(\overline{f_{\{1\}}}(u)=0\ne 1-u = \overline{f}(u,0) = \overline{f}_{\{1\}}(u)\,.\) 
\end{enumerate}
\end{example}

The survival function of a grounded function has the following property, which is applied in the proof of Theorem \ref{proppif}.

\begin{lemma}\label{lemsurvmarg}
Let \(f\colon \Xi\to \R\) be a function that is grounded and fulfils continuity condition \eqref{contboun1}. Then for \(I\subseteq \{1,\ldots,d\}\,,\) \(I\ne \emptyset\,,\) it holds that
\begin{align*}
\overline{f^I}= \overline{f}_I\,,
\end{align*}
i.e., the survival function of the upper \(I\)-marginal of \(f\) coincides with the lower \(I\)-marginal of the survival function of \(f\,.\)
\end{lemma}

\subsection{Examples}

In this section, we give various examples of functions that induce or do not induce a signed measure by \eqref{defmind}. First, we consider functions defined on \([0,1]^d\,.\) Then, we address the case of the non-compact domain \(\R_+^d\,.\) Note that a measure-inducing function on \([0,1]^d\) induces by \eqref{defmind} a finite signed measure while a measure-inducing function on \(\R_+^d\) may induce a non-finite signed measure, see Remark \ref{remindmeas}\eqref{remindmeas1}. For a general Cartesian product \(\Xi\) of intervals, similar results follow by a transformation of the domain due to Proposition \ref{propmeatraf}.

\subsubsection{Functions with domain \([0,1]^d\)}

Denote by \(\cC^n(\Xi)\) the class of functions on \(\Xi\) for which the (mixed) partial derivatives of order \(k\) exist for all \(k\leq n\,.\)
As a consequence of Theorem \ref{propbhkv}, the following result holds true.

\begin{corollary}\label{cordom}
If \(f\in \cC^d([0,1]^d)\,,\) then \(f\in \cF_{mi}^{c,l}([0,1]^d)\cap \cF_{mi}^{c,u}([0,1]^d)\,.\)
\end{corollary}

\begin{example}[Copulas, quasi-copulas, and semi-copulas]\label{exacopetc}~
\begin{enumerate}[(a)]
\item \label{exacopetc1}
Since \(d\)-variate copulas are \(d\)-increasing, every copula \(C\colon [0,1]^d\to [0,1]\) induces a probability measure \(\nu\) in the sense of \eqref{defmind}.  Since copulas are grounded on \([0,1]^d\), it holds in particular that \(\nu([0,u])=C(u)\) for all \(u\in [0,1]^d\).
 Similarly, by Sklar's Theorem, every \(d\)-dimensional distribution function \(F\colon \R^d\to [0,1]\) induces in the sense of \eqref{defmind} a probability measure \(\mu\) on \(\R^d\) through \(\mu((-\infty,x])=F(x)\) for \(x\in \R^d\,.\)  Note that distribution functions and thus copulas are \(\Delta\)-monotone functions.
\item For \(d\geq 2\,,\) a quasi-copulas is in general not measure-inducing, see \cite[Theorem 4.1]{Fernandez-2011a} for an example for \(d=2\,.\) As a three-dimensional example, the lower Fr\'{e}chet bound \(W^3\colon [0,1]^3 \to [0,1]\,,\) \((x_1,x_2,x_3)\mapsto \max\left\{\sum_{i=1}^3 x_i -2,0\right\}\) is a proper quasi-copula and does not induce a signed measure, see \cite[Theorem 2.4]{Nelsen-2010}. Hence, \(W^3\) is not the difference of two grounded \(\Delta\)-monotone functions, where one is bounded; cf. Theorems \ref{propbhkv} and \ref{thecfmicl}.
\end{enumerate}
\end{example}

The following example shows that a step function on \([0,1]^2\) with a step on the main diagonal is measure-inducing. However, the reflected function with a step on the anti-diagonal of \([0,1]^2\) does not induce a signed measure in the sense of Definition \ref{defmifun}.

\begin{example}[Mass on main diagonal/anti-diagonal]~\\
Consider the functions \(f,g\colon [0,1]^2 \to [0,1]\) defined by \(f(x,y)=\1_{\{x\geq y\}}\) and \(g(x,y)=\1_{\{x+y\geq 1\}}\,.\) Then, as a consequence of Theorem \ref{propbhkv}, \(f\) is measure-inducing because it has finite Hardy-Krause variation. But \(g\) is not measure-inducing because it has infinite Vitali variation.
\end{example}

%

\subsubsection{Functions with domain \(\R_+^d\)}

As a direct consequence of Proposition \ref{propcharmif}, right-continuous \(d\)-increasing or \(d\)-decreasing functions are measure-inducing in the sense of Definition \ref{defmifun}; cf. \cite{Tankov-2011} for a left-continuous version.

\begin{corollary}\label{cordifun}
Let \(f\colon \R_+^d\to \R\) be right-continuous. If \(f\) is \(d\)-increasing or \(d\)-decreasing, then \(f\) induces a signed measure by \eqref{defmind} and \eqref{eqfmi2}.
\end{corollary}

Due to Example \ref{exmindgb}, the \(I\)-marginals of a measure-inducing function \(f\) are in general not measure-inducing.
However, if the lower \(I\)-marginal \(f_I\) is either \(|I|\)-increasing or \(|I|\)-decreasing, then \(f_I\) is measure-inducing, see Corollary \ref{cordifun}. Hence the following result follows from the definition of \(\Delta\)-monotone and \(\Delta\)-antitone functions.

\begin{proposition}[\(\Delta\)-monotone and \(\Delta\)-antitone functions are in \(\cF_{mi}^{c,l}(\R_+^d)\)]\label{cordelmof}~\\
Let \(f\colon \R_+^d\to \R\) be a function that satisfies the continuity condition \eqref{contcond0} and \eqref{contboun1}.
\begin{enumerate}[(i)]
\item If \(f\) is \(\Delta\)-monotone, then \(f\in \cF_{mi}^{c,l}(\R_+^d)\,.\) In particular, for \(I\subseteq\{1,\ldots,d\}\,,\) \(I\ne \emptyset\,,\) \(f_I\) induces a non-negative measure by \eqref{defmind}.
\item If \(f\) is \(\Delta\)-antitone, then \(f\in \cF_{mi}^{c,l}(\R_+^d)\,.\) In particular, for \(I\subseteq\{1,\ldots,d\}\,,\) \(I\ne \emptyset\,,\) \((-1)^{|I|} f_I\) induces a non-negative measure by \eqref{defmind}.
\end{enumerate}
\end{proposition}

%
%

\begin{example}
The function \(f\colon \R_+\to \R\) defined by \(f(x)=x^2-2x\) is measure-inducing due to Proposition \ref{propcharmif} because the Jordan-decomposition of \(f=g-h\) is given by the increasing functions
\begin{align*}
g(x)&=(x^2-2x+1)\,\1_{\{x\geq 1\}}\,,&
h(x)&=-(x^2-2x+1)\,\1_{\{x<1\}} +1\,,
\end{align*}
where \(h\) is bounded on \(\R_+\,.\)
\end{example}

The following example shows that a \(d\)-increasing function \(f\) may induce a signed measure in the sense of \eqref{defmind}. However, if \(f\) is not grounded, it does not necessarily generate a signed measure \(\mu\) through \(\mu([0,x])=f(x)\,,\) \(x\in \R_+^d\,.\)

\begin{example}[\(2\)-increasing function]\label{couexa2incfun}~\\
Consider the function \(f\colon \R_+^2\to \R\) defined by \(f(x_1,x_2)=x_1x_2-x_2\,.\) Then, \(f\in \cF_{mi}^{c,l}(\Xi)\) and, in particular, \(f\) is measure-inducing in the sense of Definition \ref{defmifun}.
However, \(f\) does not generate a signed measure \(\mu\) in the sense that \(\mu([0,x])=f(x)\,,\) \(x\in \R_+^2\,,\) because, for \(x_1\in (0,1)\,,\) \(f(x_1,x_2)\to -\infty\) as \(x_2\to \infty\) and, for \(x_1 > 1\,,\) it holds that \(f(x_1,x_2)\to \infty\) as \(x_2\to \infty\,.\)
\end{example}

For \(f\in \cC^d(\Xi)\,,\) denote by \(f^{(d)}:=\frac{\partial^d f}{\partial x_1 \cdots \partial x_d}\) its mixed partial derivative of order \(d\) w.r.t. all components.
Denote by \(f_{(+)}^{(d)}\colon \Xi\to \R_+\) and \(f_{(-)}^{(d)}\colon \Xi\to \R_+\) the positive part and the negative part of \(f^{(d)}\,,\) respectively, i.e., \(f^{(d)} = f_{(+)}^{(d)} - f_{(-)}^{(d)}\,.\) The following result gives for functions in \(\cC^d(\R_+^d)\) a simple condition to be in the class \(\cF_{mi}^{c,l}(\R_+^d)\) and provides a formula for the induced sign measures.

\begin{proposition}[\(\cC^d\)-functions]\label{propcdfun}~\\
Let \(f\colon \R_+^d\to \R\) be a \(\cC^d\)-function. If for all \(I=\{i_1,\ldots,i_k\}\subseteq \{1,\ldots,d\}\,,\) \(I\ne \emptyset\,,\) at least one of the integrals \(\int_{\R^k_+} f_{I,(+)}^{(k)} \de \lambda^k\) and \(\int_{\R_+^k} f_{I,(-)}^{(k)} \de \lambda^k\) is finite, then \(f\in \cF_{mi}^{c,l}(\R_+^d)\) and
\begin{align}\label{eqcalc}
\nu_{f_I}(B) = \int_{B} f_I^{(k)} \de \lambda^k
\end{align}
for all \(B\subseteq \R_+^k\) measurable.
\end{proposition}

\section*{Appendix B: Proofs}

\setcounter{subsection}{0}
\setcounter{theorem}{0}
\renewcommand{\thesection}{B.\arabic{section}}
\renewcommand{\thesubsection}{B.\arabic{subsection}}
\renewcommand{\thelemma}{B.\arabic{lemma}}
\renewcommand{\thetheorem}{B.\arabic{theorem}}
\renewcommand{\theproposition}{B.\arabic{proposition}}
\renewcommand{\thedefinition}{B.\arabic{definition}}
\renewcommand{\theremark}{B.\arabic{remark}}

\subsection{Proofs of the results from Sections \ref{intro} and \ref{secmr}}


\begin{proof}[Proof of Theorem \ref{proppif}]
We show that 
\begin{align}\label{promtheq0}
\psi_g(f) = \sum_{I \subseteq \{1,\ldots,d\}\,, \atop I\ne \emptyset} \int_{\Xi_I} \overline{g^I}\de \nu_{f_I} + g^\emptyset f_\emptyset\,.
\end{align}
This implies \eqref{genformi1} and \eqref{genformi}, i.e., \(\psi_g(f) = \pi_f(\overline{g})\,,\) because on the one hand, for \(I\ne \emptyset\,,\) it holds that \(\overline{g^I}=\overline{g}_I\,,\) see Lemma \ref{lemsurvmarg}. On the other hand, we have that \(g^\emptyset = \lim_{x_1\uparrow b_1} \cdots \lim_{x_d\uparrow b_d} g(x_1,\ldots,x_d)=\overline{g}_\emptyset\) using that \(g\) is grounded.

The proof of \eqref{promtheq0} is given by induction over the dimension \(d\,.\) For the base case, let \(d=1\,.\) First, we consider the specific case where \(f\) is left-continuous and \(g\) is right-continuous. Then, we obtain from the left-continuity of \(f\) and the right-continuity of \(g\) by Lemma \ref{lemfmi} for all \(x,y\in \Xi_1\) that
\begin{align}\label{promtheq1}
f(x)&=\nu_f(\Xi_1\cap [a_1,x))+f_\emptyset\,, ~~~ \text{and}\\
\label{promtheq2}\nu_g((y,b_1] \cap \Xi_1)&= g^\emptyset-g(y) = \overline{g}(y)\,,
\end{align}
applying the continuity conditions \eqref{contboun1} and \eqref{contboun2}, respectively, where the last equality holds by definition of the survival function in \eqref{defsurvfun}. Then \eqref{promtheq0} follows from
\begin{align}
\nonumber \psi_g(f)=\int_{\Xi_1} f(x) \de \nu_g(x) &= \int_{\Xi_1}\left(\nu_f(\Xi_1\cap [a_1,x))+f_\emptyset\right) \de \nu_g(x) \\
\label{promtheq3}&= \int_{\Xi_1}\left(\int_{\Xi_1} \1_{\{a_1\leq y<x\}}(y) \de \nu_f(y)\right) \de \nu_g(x)+f_\emptyset\cdot \nu_g(\Xi_1)\\
\label{promtheq3a}&= \int_{\Xi_1}\left(\int_{\Xi_1} \1_{\{a_1\leq y<x\}}(x) \de \nu_g(x)\right) \de \nu_f(y)+f_\emptyset g^\emptyset\\
\nonumber &= \int_{\Xi_1} \nu_g((y,b_1] \cap \Xi_1) \de \nu_f(y) + f_\emptyset g^\emptyset\\
\nonumber &= \int_{\Xi_1} \overline{g}(y) \de \nu_f(y)+ g^\emptyset f_\emptyset
\end{align}
where the second and sixth equality follow from \eqref{promtheq1} and \eqref{promtheq2}, respectively. For the third equality, we use that \(f_\emptyset\) is finite and that \(g\) is bounded and thus, by Remark \ref{remindmeas}\eqref{eqremindmeas2a}, \(\nu_g(\Xi)\) is finite.
For the fourth equality, we decompose each integral into the sum of two integrals w.r.t. \(\sigma\)-finite measures, compare Remark \ref{remindmeas}, and then apply Fubini's theorem, see, e.g., \cite[Theorem 3.7.5]{Benedetto-2009}. Further, note that \(\nu_g(\Xi_1)=g^\emptyset-g_\emptyset = g^\emptyset\) using that \(g\) is grounded.

For the induction step, assume that \eqref{promtheq0} holds true for all marginals of \(f\) and \(h\) up to dimension \(d-1\,,\) i.e.,
\begin{align}\label{eqproproppif0}
\int_{\Xi_I} f_I(x) \de \nu_{g^I}(x) = \sum_{J\subseteq I \atop J\ne \emptyset} \int_{\Xi_J} \overline{g^J}(u)\de \eta_{f_J}(u) + g^\emptyset f_\emptyset\,.
\end{align}
for all \(I\subsetneq \{1,\ldots,d\}\) using that \((g^I)^J =g^J\) and \((f_I)_J=f_J\) for \(J\subseteq I\) and, in particular, \((g^I)^\emptyset = g^\emptyset\) and \((f_I)_\emptyset=f_\emptyset\,.\)
Note that \(g^I\) is measure-inducing due to Lemma \ref{lemmargdisgrou}.
 Since \(f\) is left-continuous and measure-inducing, we obtain from Lemma \ref{lemfmi}\eqref{lemfmi1} and from the definition of the lower \(I\)-marginal in \eqref{defqmarg} that
\begin{align}\label{eqproproppif1}
\begin{split}
\nu_f\left(\Xi\cap \bigtimes_{i=1}^d [a_i,x_i)\right)& = \lim_{\varepsilon_i \downarrow 0 \atop i=1,\ldots,d} \nu_f\left( \bigtimes_{i=1}^d [a_i+\varepsilon_i,x_i)\right) \\
&= \sum_{k=0}^d \sum_{I\subseteq \{1,\ldots,d\} \atop I=\{i_1,\ldots,i_k\}} (-1)^{d-|I|} f_I(x_{i_1},\ldots,x_{i_k})\\
& = f(x_1,\ldots,x_d) + \sum_{k=0}^{d-1} \sum_{I\subsetneq \{1,\ldots,d\}\atop I=\{i_1,\ldots,i_k\}} (-1)^{d-|I|} f_I(x_{i_1},\ldots,x_{i_k})\,.
\end{split}
\end{align}
Similarly, since \(h\) is right-continuous and measure-inducing, Proposition \ref{lemfmi}\eqref{lemfmi2} as well as the definition of the survival function in \eqref{defsurvfun} imply
\begin{align}\label{eqproproppif2}
\nu_g\left(\Xi\cap \bigtimes_{i=1}^d (x_i,b_i]\right) = \lim_{\varepsilon_i \downarrow 0\atop i=1,\ldots,d}\nu_g\left( \bigtimes_{i=1}^d (x_i,b_i-\varepsilon_i]\right)
= \overline{g}(x_1,\ldots,x_d)\,.
\end{align}
We now obtain on the one hand by an application of Fubini's theorem to \(\sigma\)-finite signed measures that
\begin{align}\label{eqproproppif3}
\begin{split}
&\int_{\Xi} \nu_f\left(\Xi\cap \bigtimes_{i=1}^d [a_i,x_i)\right) \de \nu_g(x_1,\ldots,x_d)\\
&= \int_{\Xi}\left(\int_{\Xi} \1_{\{a_1\leq y_1<x_1\}} \cdots \1_{\{a_d\leq y_d<x_d\}} \de \nu_f(y_1,\ldots,y_d)\right) \de \nu_g(x_1,\ldots,x_d)\\
&= \int_{\Xi}\left(\int_{\Xi} \1_{\{a_1\leq y_1<x_1\}} \cdots \1_{\{a_d\leq y_d<x_d\}} \de \nu_g(x_1,\ldots,x_d)\right) \de \nu_f(y_1,\ldots,y_d)\\
&= \int_{\Xi} \nu_g\left(\Xi\cap \bigtimes_{i=1}^d (y_i,b_i]\right)\de \nu_f(y_1,\ldots,y_d)\\
&= \int_{\Xi} \overline{g}(y) \de \nu_f(y)\,,
\end{split}
\end{align}
where the last equality follows from \eqref{eqproproppif2}.
On the other hand, it holds that
\begin{align}
\nonumber &-\int_{\Xi}  \sum_{k=0}^{d-1} \sum_{I\subsetneq \{1,\ldots,d\}\atop I=\{i_1,\ldots,i_k\}} (-1)^{d-|I|} f_I(x_{i_1},\ldots,x_{i_k}) \de \nu_g(x_1,\ldots,x_d)\\
\nonumber &=- (-1)^d f_\emptyset g^\emptyset- \sum_{k=1}^{d-1} \sum_{I\subsetneq \{1,\ldots,d\}\atop I=\{i_1,\ldots,i_k\}} (-1)^{d-|I|} \int_{\Xi_I} f_I(x) \de \nu_{g^I}(x) \\
\label{eqproproppif4}&= (-1)^{d-1} g^\emptyset f_\emptyset  - \sum_{k=1}^{d-1} \sum_{I\subsetneq\{1,\ldots,d\}\atop I=\{i_1,\ldots,i_k\}} (-1)^{d-|I|} \left(\sum_{J\subseteq I\atop J\ne \emptyset} \int_{\Xi_J} \overline{h^J}(u) \de \nu_{f_J}(u)+ g^\emptyset f_\emptyset\right)\\
\label{eqproproppif4a}&= \sum_{k=1}^{d-1}\sum_{I\subsetneq \{1,\ldots,d\}\atop I=\{i_1,\ldots,i_k\}}\int_{\Xi_I} \overline{g^I}(x)\de \nu_{f_I}(x) + g^\emptyset f_\emptyset\,,
\end{align}
where we apply for the second equality the induction hypothesis in \eqref{eqproproppif0}. To show the first equality, we obtain for \(I=\emptyset\) that
\begin{align*}
\int_{\Xi} (-1)^d f_\emptyset \de \nu_g(x_1,\ldots,x_d)= (-1)^d f_\emptyset \cdot \nu_g(\Xi) = (-1)^d f_\emptyset g^\emptyset \,,
\end{align*}
using that \(f_\emptyset\) and \(\nu_g(\Xi)=g^\emptyset\) are finite. For \(I=\{i_1,\ldots,i_k\}\subsetneq \{1,\ldots,d\}\,,\) \(k\geq 1\,,\) denote for \(I^c:=\{1,\ldots,d\}\setminus I\) by \((\nu_g^+)^{I|I^c}\) and \((\nu_g^-)^{I|I^c}\) the conditional measure of \(\nu_g^+\) and \(\nu_g^-\,,\) respectively, w.r.t. the components of \(I\) given the components of \(I^c\,.\) Then, by the disintegration theorem applied to the Jordan decomposition of \(\nu_g=\nu_g^+-\nu_g^-\,,\) it follows that
\begin{align*}
&\int_{\Xi} f_I(x_{i_1},\ldots,x_{i_k}) \de \nu_g(x_1,\ldots,x_d)\\
&= \int_{\Xi} f_I(x_{i_1},\ldots,x_{i_k}) \de \nu_g^+(x_1,\ldots,x_d) - \int_{\Xi} f_I(x_{i_1},\ldots,x_{i_k}) \de \nu_g^-(x_1,\ldots,x_d)\\
&= \int_{\Xi_{I^c}}\int_{\Xi_I} f_I(x_I) \de (\nu_g^+)^{I|I^c}(x_I|x_{I^c}) \de (\nu_g^+)^{I^c}
-  \int_{\Xi_{I^c}}\int_{\Xi_I}  f_I(x_I) \de (\nu_g^-)^{I|I^c}(x_I|x_{I^c}) \de (\nu_g^-)^{I^c}\\
&= \int_{\Xi_I} f_I(x) \de (\nu_g^+)^I  -  \int_{\R_+^{|I|}} f_I(x) \de (\nu_g^-)^I\\
& = \int_{\Xi_I} f_I(x) \de \nu_g^I\\
& =  \int_{\Xi_I} f_I(x) \de \nu_{g^I}(x)\,,
\end{align*}
where the last equality is given by Lemma \ref{lemmargdisgrou} using that \(g\in \cF_{mi}^c(\Xi)\) is grounded and bounded.
 
 For the equality in \eqref{eqproproppif4a}, we first sum over all terms \(g^\emptyset f_\emptyset\) in \eqref{eqproproppif4} and obtain
 \begin{align*}
 (-1)^{d-1} g^\emptyset f_\emptyset - \sum_{k=1}^{d-1} (-1)^{d-k} \binom d k g^\emptyset f_\emptyset &= \sum_{k=0}^{d-1} (-1)^{d-k-1} \binom d k g^\emptyset f_\emptyset \\
 &= \underbrace{\sum_{k=0}^{d} (-1)^{d-k-1} \binom d k g^\emptyset f_\emptyset}_{=0} - (-1)^{d-d-1} \binom d d g^\emptyset f_\emptyset = g^\emptyset f_\emptyset
\end{align*}  
 applying a well-known identity for the sum of alternating binomial coefficients.
For simplifying the alternating sum of the integrals in \eqref{eqproproppif4},  we determine for fixed \(J\subsetneq \{1,\ldots,d\}\,,\) \(|J|=:m < d\,,\) the number of proper subsets \(I\) of \(\{1,\ldots,d\}\) with \(|I|=k\) for some \(k\in \{m,m+1,\ldots,d-1\}\,,\) such that \(I\) is also a superset of \(J\,,\) i.e., we determine \(N_k:=\#\{I\colon |I|=k\,,  J\subseteq I\subsetneq \{1,\ldots,d\}\}\,.\) Subtracting the set \(J\,,\) we have that \(N_k=\#\{I\colon |I|=k-m\,, I\subsetneq \{1,\ldots,d\}\setminus J\} = \binom {d-m}{k-m}\,.\) So taking the alternating sum, for each fixed marginal \(J\,,\) the signed number of integrals w.r.t. the marginal \(J\) in \eqref{eqproproppif4} sums up to
\begin{align*}
\sum_{k=m}^{d-1}(-1)^{d-k}N_k =\sum_{k=m}^{d-1}(-1)^{d-k} \binom{d-m}{k-m} &= \sum_{k=0}^{d-1-m} (-1)^{d-k-m} \binom{d-m}k \\
&=  \sum_{k=0}^{d-m} (-1)^{d-k-m} \binom{d-m}k - \binom{d-m}{d-m} = 0-1=-1
\end{align*} 
applying again the identity for the sum of alternating binomial coefficients.

Combining \eqref{eqproproppif1}, \eqref{eqproproppif3}, and \eqref{eqproproppif4},
we finally obtain
\begin{align*}
\psi_g(f) &=\int_{\Xi} f(x) \de \nu_g(x)\\
&= \int_{\Xi} \nu_f \left(\Xi\cap \bigtimes_{i=1}^d [a_i,x_i)\right) \de \nu_g(x_1,\ldots,x_d)\\
&~~- \int_{\Xi}  \sum_{k=0}^{d-1} \sum_{I\subsetneq \{1,\ldots,d\}\atop I=\{i_1,\ldots,i_k\}} (-1)^{d-|I|} f_I(x_{i_1},\ldots,x_{i_k}) \de \nu_g(x_1,\ldots,x_d)\\
&= \int_{\Xi} \overline{g}(x_1,\ldots,x_d)\de \nu_f(x_1,\ldots,x_d) + \sum_{k=1}^{d-1}\sum_{I\subsetneq \{1,\ldots,d\}\atop I=\{i_1,\ldots,i_k\}}\int_{\Xi_I} \overline{g^I}(x)\de \nu_{f_I}(x) + g^\emptyset f_\emptyset\\
\nonumber &= \sum_{I\subseteq \{1,\ldots,d\} \atop I\ne \emptyset} \int_{\Xi_I} \overline{g^I}(u)\de \nu_{f_I}(u) + g^\emptyset f_\emptyset
\end{align*}
which proves \eqref{promtheq0} in the case that \(f\) is left-continuous and \(g\) is right-continuous.

Before we consider the general case, assume that \(f\) is right-continuous and \(g\) is left-continuous. Then \eqref{promtheq0} follows in the same way, taking into account the following changes:  Due to Lemma \ref{lemfmi}\eqref{lemfmi2}, substitute for the base case of the induction in \eqref{promtheq1} and \eqref{promtheq2} \(\nu_f(\Xi_1\cap [a_1,x])\) and \(\nu_g([y,b_1]\cap \Xi_1)\) for \(\nu_f(\Xi_1\cap [a_1,x))\) and \(\nu_g((y,b_1]\cap \Xi_1)\,,\) respectively. Then, the set of the indicator function in \eqref{promtheq3} and \eqref{promtheq3a} changes from \(\{a_1 \leq y < x\}\) to \(\{a_1\leq y \leq x\}\,.\)
For the induction step, substitute in \eqref{eqproproppif1}, \eqref{eqproproppif2}, and \eqref{eqproproppif3} \(\nu_f\left(\Xi\cap \bigtimes_{i=1}^d [a_i,x_i]\right)\) and \(\nu_g\left(\Xi\cap \bigtimes_{i=1}^d [x_i,b_i]\right)\) for \(\nu_f\left(\Xi\cap \bigtimes_{i=1}^d [a_i,x_i)\right)\) and \(\nu_g\left(\Xi\cap \bigtimes_{i=1}^d (x_i,b_i]\right)\,,\) respectively, using that \(f\) and \(g\) fulfil the continuity conditions \eqref{contboun1} and \eqref{contboun2}.

Now, consider the general case where \(f\) and \(g\) have no common discontinuities on the same side of each point. If \(f\) has a jump discontinuity on the left/right side of \(x\) in the \(i\)th component, then, due to Assumption \ref{asscont}\eqref{asscont1}, \(f\) is right-/left-continuous in \(x\) in the \(i\)-th component, and, due to \eqref{defcomdiscy}, \(g\) and thus \(\overline{g}\) are left-/right-continuous at \(x\) in the \(i\)th component. In this case, choose the interval \([a_i,x_i]\)/\([a_i,x_i)\) for the \(i\)th component in \eqref{eqproproppif1} and \([x_i,b_i]\)/\((x_i,b_i]\) for the \(i\)th component in \eqref{eqproproppif2}. The rest of the proof follows similarly.
\end{proof}

\begin{proof}[Proof of Proposition \ref{lemmargtraf}]
We show for \(I=\{1,\ldots,d\}\) that
\begin{align}\label{eqpromt1}
\int_{[0,1]^{|I|}} g_I(u) \de \nu_{(f\, \circ(F_1^{[-1]},\ldots,F_d^{[-1]}))_I}(u)=
\int_{\Xi_I} (g \circ (F_1,\ldots,F_d))_I(x) \de \nu_{f_I}(x)\,.
\end{align}
Since, for \( I\subsetneq \{1,\ldots,d\}\,,\) \(I\ne \emptyset\,,\) the integral transformation in \eqref{eqpromt1} follows similarly, the statement follows from \(g_\emptyset=(g\circ (F_1,\ldots,F_d))_\emptyset\) and \(f(F_1^{[-1]}(a_1),\ldots,F_d^{[-1]}(a_d)) = f(0,\ldots,0)=f_\emptyset\) noting that \(F_i^{[-1]}(a_i)=0\) by definition of the generalized inverse in \eqref{defgeninvy}.

We prove \eqref{eqpromt1} for \(I=\{1,\ldots,d\}\) by algebraic induction.
Since \(f_I\) is measure-inducing for all \(I\subseteq\{1,\ldots,d\}\,,\) \(I\ne \emptyset\,,\) the left-hand limits of \(f\) exist, compare Theorem \ref{thecfmicl}. Thus, we may assume due to Lemma \ref{lemunimifun} that \(f\) is left-continuous, applying that \(f\) satisfies continuity condition \eqref{contboun1}.
Assume first that \(g\) is an indicator function given by \(g(u)=\1_{\bigtimes_{i=1}^d [\alpha_i,\beta_i)}(u)\,,\) \(u=(u_1,\ldots,u_d)\in [0,1]^d\,,\) where \(\alpha_i,\beta_i\in (0,1]\) such that \(\alpha_i<\beta_i\) for all \(i\in \{1,\ldots,d\}\,.\)
We show that
\begin{align}\label{eqpromt2}
\int_{[0,1]^d} g(u) \de \nu_{f\,\circ (F_1^{[-1]},\ldots,F_d^{[-1]})} (u) = \int_{\Xi} g(F_1(x_1),\ldots,F_d(x_d)) \de \nu_f(x_1,\ldots,x_d)  \,.
\end{align}
It holds that \(\alpha_i\leq F_i(x_i)<\beta_i\) if and only if \(F_i^{[-1]}(\alpha_i)\leq x_i < F_i^{[-1]}(\beta_i)\,,\) using that \(F_i(a_i)=0\) and \(F_i(b_i)=1\,,\) see, e.g., \cite[Proposition 1(e)]{Winter-1997}. Hence,
we obtain for the right-hand side of \eqref{eqpromt2} that
\begin{align*}
\int_{\Xi} g(F_1(x_1),\ldots,F_d(x_d)) \de \nu_f(x_1,\ldots,x_d) 
&= \int_{\Xi} \1_{\bigtimes_{i=1}^d [\alpha_i,\beta_i)}(F_1(x_1),\ldots,F_d(x_d)) \de \nu_f(x)\\
&= \int_{\Xi} \1_{\bigtimes_{i=1}^d [F_i^{[-1]}(\alpha_i),F_i^{[-1]}(\beta_i))}(x) \de \nu_f(x)\\
&= \nu_f\left(\bigtimes_{i=1}^d [F_i^{[-1]}(\alpha_i),F_i^{[-1]}(\beta_i))\right).
\end{align*}
For the left-hand side of \eqref{eqpromt2}, we obtain for \(\varepsilon_i=\beta_i-\alpha_i\) and \(\delta_i:=F_i^{[-1]}(\beta_i)-F_i^{[-1]}(\alpha_i)\,,\) \(i=1,\ldots,d\,,\) that
\begin{align*}
\int_{[0,1]^d} g(u)\de \nu_{f\,\circ (F_1^{[-1]},\ldots,F_d^{[-1]})}(u) &= \nu_{f\,\circ(F_1^{[-1]},\ldots,F_d^{[-1]})}\left(\bigtimes_{i=1}^d [\alpha_i,\beta_i)\right) \\
&= \Delta_{\varepsilon_1}^1 \cdots \Delta_{\varepsilon_d}^d \,f\,\circ (F_1^{[-1]},\ldots,F_d^{[-1]})(\alpha_1,\ldots,\alpha_d)\\
&= \Delta_{\delta_1}^1 \cdots \Delta_{\delta_d}^d \,f \left((F_1^{[-1]}(\alpha_1),\ldots,F_d^{[-1]}(\alpha_d)\right))\\
&= \nu_f\left(\bigtimes_{i=1}^d [F_i^{[-1]}(\alpha_i),F_i^{[-1]}(\beta_i))\right)\,.
\end{align*}
applying Proposition \ref{lemfmi}\eqref{lemfmi1} where we make use of the assumption that \(f\) and thus \(f\circ (F_1^{[-1]},\ldots,F_d^{[-1]})\) are left-continuous.
Hence, \eqref{eqpromt2} also holds true whenever \(h\) is an indicator function of a Cartesian product of Borel sets, see, e.g., \cite[Theorem 10.3]{Billingsley-1995}.

By a standard approximation argument, \eqref{eqpromt2} also holds true if \(h\) is a \(\nu_f^+\)- and \(\nu_f^-\)-integrable and thus a \(\nu_f\)-integrable function which concludes the proof.
\end{proof}

\subsection{Proofs of the results from Sections \ref{secvonv} and \ref{appstoord}}

\begin{proof}[Proof of Theorem \ref{theconvfin}:]
Since \(f_I\,,\) \(I\subseteq \{1,\ldots,d\}\,,\) is defined on a compact set, it induces a finite signed measure, see Remark \ref{remindmeas}\eqref{remindmeas1}. 
As a consequence of the assumptions, it holds that \(|g|\leq M\) and thus \(|\overline{h}_I|\leq d^2 M.\)
Since \(g\) is the limit of measurable functions, it follows that \(g\) and thus \(\overline{g}_I\) are measurable.
Using that \((\overline{g_n})_I(x) \to \overline{g}_I(x)\) for all \(x\in [0,1]^d\,,\) the dominated convergence theorem yields for \(I\subseteq \{1,\ldots,d\}\,,\) \(I\ne \emptyset\,,\) that
\begin{align}\label{prth37}
\begin{split}
\int_{[0,1]^{|I|}} (\overline{g_n})_I\de \nu_{f_I} &= \int_{[0,1]^{|I|}} (\overline{g_n})_I\de \nu_{f_I}^+- \int_{[0,1]^{|I|}} (\overline{g_n})_I\de \nu_{f_I}^- \\
& \xrightarrow{n\to \infty} \int_{[0,1]^{|I|}} \overline{g}_I\de \nu_{f_I}^+ - \int_{[0,1]^{|I|}} \overline{g}_I\de \nu_{f_I}^-
= \int_{[0,1]^{|I|}} \overline{g}_I\de \nu_{f_I}\,,
\end{split}
\end{align}
using that \(\overline{g}_I\) is bounded and \(\nu_{f_I}^+\) as well as \(\nu_{f_I}^-\) are finite.
Since \((\overline{g_n})_\emptyset = g_n(1,\ldots,1)\to g(1,\ldots,1)=\overline{g}_\emptyset\,,\) \eqref{prth37} implies that
\begin{align*}
\psi_{g_n}(f) = \pi_f(\overline{g_n}) &= \sum_{I\subseteq \{1,\ldots,d\}\atop I\ne \emptyset} \int_{[0,1]^{|I|}} (\overline{g_n})_I \de \nu_{f_I} + (\overline{g_n})_\emptyset f_\emptyset\\
& \xrightarrow{n\to \infty} \sum_{I\subseteq \{1,\ldots,d\}\atop I\ne \emptyset} \int_{[0,1]^{|I|}} \overline{g}_I \de \nu_{f_I} + \overline{g}_\emptyset f_\emptyset = \pi_f(\overline{g})
\end{align*}
where we apply Theorem \ref{proppif} for the first equality.
\end{proof}

\begin{proof}[Proof of Theorem \ref{theconvthsec}:]
Let \(U\) be a uniformly on \((0,1)\) distributed random variable defined on a non-atomic probability space \((\Omega,\cA,P)\,.\) For \(1\leq i \leq d\,,\) define \(X_i:=F_i^{[-1]}(U)\,.\) For all \(I=\{i_1,\ldots,i_k\}\subseteq\{1,\ldots,d\}\,,\) \(I\ne \emptyset\,,\) the vector \((X_{i_1},\ldots,X_{i_k})\) is a comonotonic random vector with distribution function given by
\begin{align}\label{eqtheconvthsec0}
F_{X_{i_1},\ldots,X_{i_k}}(x)=M^k(F_{i_1}(x_1),\ldots,F_{i_k}(x_k))
\end{align}
for \(x=(x_1,\ldots,x_k)\in \R^k\,,\) see, e.g., \cite[Section 5.1.6]{Embrechts-2015}, where \(M^k\) is the upper Fr\'{e}chet copula defined by 
\begin{align}\label{defuppfrecop}
M^k(u) :=\min_{1\leq i \leq k}\{u_i\} ~~~\text{for } u=(u_1,\ldots,u_k)\in [0,1]^k\,.
\end{align}
It follows by assumption \eqref{theconvthsec3} inductively for \(|I|=1,\ldots,d\,,\) \(I=\{i_1,\ldots,i_k\}\subseteq\{1,\ldots,d\}\,,\) that all the integrals
\begin{align}
\nonumber &\int_0^1 f_I\left(F_{i_1}^{[-1]}(u),\ldots,F_{i_k}^{[-1]}(u)\right) \de u \\
\nonumber &= \int_\Omega f_I\left(F_{i_1}^{[-1]}(U),\ldots,F_{i_k}^{[-1]}(U)\right) \de P\\ 
\nonumber &= \int_{\R^d} f_I(x) \de P^{(X_{i_1},\ldots,X_{i_k})}(x) \\
\nonumber &= \int_{\R^d} f_I(x_1,\ldots,x_k) \de M^k(F_{i_1}(x_1),\ldots,F_{i_k}(x_k))\\
\nonumber &= \sum_{l=1}^k \sum_{J\subseteq I\atop J=\{j_1,\ldots,j_l\}} \int_{\R^{|J|}} (\overline{M^k})_J (F_{j_1}(x_1),\ldots,F_{j_l}(x_l)) \de \nu_{f_J}(x_1,\ldots,x_l) + f_\emptyset\\
\label{eqtheconvthsec1}&= f_\emptyset + \sum_{l=1}^k \sum_{J\subseteq I\atop J=\{j_1,\ldots,j_l\}} \bigg[\int_{\R^{|J|}} \left( 1- \max_{j\in J}\{F_{j}(x_j)\}\right) \de \nu_{f_J}^+(x_{j_1},\ldots,x_{j_l}) \\
\label{eqtheconvthsec2}&~~~~~~~~~~~~~~~~~~~~~~~~~~~~~~~~~~~~~~~~~  - \int_{\R^{|J|}} \left( 1- \max_{j\in J}\{F_{j}(x_j)\}\right) \de \nu_{f_J}^-(x_{j_1},\ldots,x_{j_l})\bigg]
\end{align}
exist, where the first equality holds true since \(U\) is uniformly distributed on \((0,1)\,.\)
The second equality follows with the definition of \(X_i\,,\) \(i=1,\ldots,d\,.\)
The third equality holds true due to \eqref{eqtheconvthsec0}.
The fourth equality is a consequence of Proposition \ref{proppif} using that \(f_I\) is left-continuous and \(M^k\circ(F_{i_1},\ldots,F_{i_k})\) is right-continuous. Note that \((\overline{M^k})_\emptyset=1\) and \(\left(\overline{M^k\circ (F_{i_1},\ldots,F_{i_k})}\right)_J=(\overline{M^k})_J\circ(F_{j_1},\ldots,F_{j_l})\) because \(F_{1},\ldots,F_{d}\) are increasing for \(J=\{j_1,\ldots,j_l\}\subseteq I\,,\) \(J\ne \emptyset\,.\)
For the last equality, we apply the Jordan decomposition of the signed measure \(\nu_{f_J}\) and use that \(\overline{M^l}(u)=1-\max_{1\leq i \leq l}\{u_i\}\) for \(u=(u_1,\ldots,u_l)\in [0,1]^l\,.\)

Let \(I=\{i_1,\ldots,i_k\}\subseteq \{1,\ldots,d\}\,,\) \(I\ne \emptyset\,.\)
Due to assumption \eqref{theconvthsec2}, it holds that
\begin{align*}
\left| (\overline{g_n})_I(x_1,\ldots,x_k)\right| \leq \alpha (1-\max\{F_{i_1}(x_1),\ldots,F_{i_k}(x_k)\}) ~~~\text{for all } (x_1,\ldots,x_k)\in \R_+^k\,,
\end{align*}
where the right-hand side is by \eqref{eqtheconvthsec1} and \eqref{eqtheconvthsec2} \(\nu_{f_I}^+\)- and \(\nu_{f_I}^-\)-integrable.

Due to assumption \eqref{theconvthsec1}, it follows for all \(x\in \R^d\) that \(\overline{g_n}(x) \to \overline{g}(x)\) and thus \((\overline{g_n})_I(x) \to \overline{g}_I(x)\) for all \(I\subseteq\{1,\ldots,d\}\,,\) \(I\ne \emptyset\,.\)
Hence, we obtain from the dominated convergence theorem for all \(I=\{i_1,\ldots,i_k\}\subseteq \{1,\ldots,d\}\,,\) \(I\ne \emptyset\,,\) that
\begin{align}\label{theconvthsec4}
\begin{split}
\int_{\R_+^{|I|}} (\overline{g_n})_I(x) \de \nu_{f_I}^+(x) &\to \int_{\R_+^{|I|}} \overline{g}_I(x) \de \nu_{f_I}^+(x) ~~~ \text{and} \\ \int_{\R_+^{|I|}} (\overline{g_n})_I(x) \de \nu_{f_I}^-(x) &\to \int_{\R_+^{|I|}} \overline{g}_I(x) \de \nu_{f_I}^-(x)
\end{split}
\end{align}
as \(n\to \infty\,.\)

By assumption \eqref{theconvthsec1b}, it holds that \((\overline{g_n})_\emptyset = \overline{g_n^\emptyset} \to \overline{g^\emptyset}=(\overline{g})_\emptyset\,.\) This implies with \eqref{theconvthsec4} that
\begin{align*}
\psi_{g_n}(f) = \pi_f(\overline{g_n}) \to \pi_f(\overline{h})
\end{align*}
where the equality follows from Theorem \ref{proppif}.
\end{proof}

\begin{proof}[Proof of Lemma \ref{propsvsemcopbou}:]
For \(u=(u_1,\ldots,u_d)\in [0,1]^d\,,\) assume w.l.o.g. that \(u_1=\max_{1\leq i \leq d}\{u_i\}\,.\)
By definition of the survival function in \eqref{defsurvfun}, it holds that
\begin{align}
\nonumber \overline{S}(u) &= \sum_{k=0}^d \sum_{I\subseteq \{1,\ldots,d\}\atop I=\{i_1,\ldots,i_k\}} (-1)^{|I|} S^I(u_{i_1},\ldots,u_{i_k}) \\
\label{eqpropsvsemcopbou2}&= 1-u_1 + \sum_{k=1}^{d-1} \sum_{J\subseteq \{2,\ldots,d\} \atop J=\{j_1,\ldots,j_k\}} (-1)^{|J|} S^{J\cup \{1\}} (1,u_{j_1},\ldots,u_{j_k}) \\
\label{eqpropsvsemcopbou3}& ~~~~~~~~~~~~+\sum_{k=1}^{d-1} \sum_{J\subseteq \{2,\ldots,d\} \atop J=\{j_1,\ldots,j_k\}} (-1)^{|J|+1} S^{J\cup \{1\}} (u_1,u_{j_1},\ldots,u_{j_k})\\
\label{eqpropsvsemcopbou1}&\begin{cases}
\leq 1-u_1 + L \,2^{d-2} (1-u_1) \\
\geq 1-u_1-L\, 2^{d-2} (1-u_1)
\end{cases}
\end{align}
where we use for the second equality the definition of the upper \(I\)-marginal and that \(S\) has uniform marginals. To show the inequalities in \eqref{eqpropsvsemcopbou1}, we obtain for \(J=\{j_1,\ldots,j_k\}\subseteq \{2,\ldots,d\}\) in the case that \(|J|\) is even
\begin{align}\label{eqpropsvsemcopbou4}
\begin{split}
&(-1)^{|J|} S^{J\cup \{1\}} (1,u_{j_1},\ldots,u_{j_k}) + (-1)^{|J|+1} S^{J\cup \{1\}} (u_1,u_{j_1},\ldots,u_{j_k}) \\
&= S^{J\cup \{1\}} (1,u_{j_1},\ldots,u_{j_k}) - S^{J\cup \{1\}} (u_1,u_{j_1},\ldots,u_{j_k}) \\
&\begin{cases} \leq L\, (1-u_1) \\ \geq 0\,, \end{cases}
\end{split}
\end{align}
applying that \(S\) is \(L\)-Lipschitz continuous and increasing, respectively. In the case that \(|J|\) is odd, it follows similarly that
\begin{align}\label{eqpropsvsemcopbou5}
\begin{split}
&(-1)^{|J|} S^{J\cup \{1\}} (1,u_{j_1},\ldots,u_{j_k}) + (-1)^{|J|+1} S^{J\cup \{1\}} (u_1,u_{j_1},\ldots,u_{j_k}) \\
&= -S^{J\cup \{1\}} (1,u_{j_1},\ldots,u_{j_k}) + S^{J\cup \{1\}} (u_1,u_{j_1},\ldots,u_{j_k}) \\
&\begin{cases} \leq 0 \\\geq -L\, (1-u_1)  \,.\end{cases}
\end{split}
\end{align}
Since each double sum in \eqref{eqpropsvsemcopbou2} and \eqref{eqpropsvsemcopbou3} has in total \(2^{d-1}\) summands for half of which \(|J|\) is even/odd, we obtain from \eqref{eqpropsvsemcopbou4} and \eqref{eqpropsvsemcopbou5} the inequalities in \eqref{eqpropsvsemcopbou1}. This yields
\begin{align*}
\overline{S}(u) &\leq 1-u_1+L \,2^{d-2} (1-u_1) =  (L\, 2^{d-2}+1) (1-u_1) ~~~\text{and} \\
-\overline{S}(u) & \leq -(1-u_1) + L\, 2^{d-2} (1-u_1) =  (L\, 2^{d-2}-1) (1-u_1) \leq  (L\, 2^{d-2}+1) (1-u_1)\,,
\end{align*}
which proves the statement.
\end{proof}

\begin{proof}[Proof of Corollary \ref{corconsemcopres}:]
We verify the assumptions of Theorem \ref{theconvthsec} choosing for \(F_1,\ldots,F_d\) the distribution function of a uniform distribution on \([0,1]\,,\) i.e., \(F_i(x)=x\) for all \(x\in [0,1]\) and \(i\in \{1,\ldots,d\}\,.\)

 Since \(F_{i,n}\) is the distribution function of a distribution on \((0,1)\) with finite support and since \(S\) is continuous and grounded, it follows that \(g_n\colon \R_+^d \to [0,1]\) defined by \(g_n:=S\circ (F_{1,n},\ldots,F_{d,n})\) is right-continuous, measure-inducing, and grounded. Condition \eqref{corconsemcopres1} and the continuity of \(S\) imply that \(h_n\to h:=S\) pointwise, i.e., assumption \eqref{theconvthsec1} of Theorem \ref{theconvthsec} holds true.
Since \(S\) is a continuous semi-copula and \(F_{i,n}(0)=0\) for all \(i\) and \(n\,,\) it follows that \(g_n^\emptyset=1 = g^\emptyset\) for all \(n\,.\)
So, assumption \eqref{theconvthsec1b} of Theorem \ref{theconvthsec} is fulfilled. 

Now, fix \(n\in \N\) and let \((x_1,\ldots,x_d)\in [0,1]^d\,.\) 
Since \(F_{i,n}\) is the distribution function of a distribution on \((0,1)\) with finite support, there exists \(\varepsilon>0\) such that \(F_{i,n}(v)\geq v\) (or even such that \(F_{i,n}(v)=1\)) for all \(v\in [1-\varepsilon,1]\) and for all \(i\in \{1,\ldots,d\}\,.\) 
Define \(\beta:=  (L\,2^{d-2}+1)\) and \(\alpha:=\beta/\varepsilon\,.\)\\
If \(\max_{i=1,\ldots,d}\{x_i\}\geq 1-\varepsilon\,,\) then it holds that
\begin{align}\label{corveras1}
|\overline{S}(F_{1,n}(x_1),\ldots,F_{d,n}(x_d))| \leq \beta \,(1-\max_{1\leq i \leq d}\{F_{i,n}(x_i)\}) \leq \beta \,(1-\max_{1\leq i \leq d}\{x_i\})\leq \alpha \,(1-\max_{1\leq i \leq d}\{x_i\})\,,
\end{align}
where the first inequality follows from Lemma \ref{propsvsemcopbou}. The second inequality holds true
because \(\max_i\{F_{i,n}(x_i)\}\geq \max_i\{x_i\}\) whenever there exists \(x_i\) such that \(x_i\in [1-\varepsilon,1]\,.\) The last inequality is fulfilled because \(\varepsilon\in (0,1)\) and \(\max_i\{x_i\}\leq 1\,.\)\\
If \(\max_{i=1,\ldots,d}\{x_i\} < 1-\varepsilon\,,\) then we obtain 
\begin{align}\label{corveras2}
|\overline{S}(F_{1,n}(x_1),\ldots,F_{d,n}(x_d))| \leq \beta \,(1-\max_{1\leq i \leq d}\{F_{i,n}(x_i)\}) \leq \alpha \varepsilon \leq \alpha \,(1-\max_{1\leq i \leq d}\{x_i\})\,,
\end{align}
where the first inequality follows from Lemma \ref{propsvsemcopbou}. The second inequality holds true because \(F_{i,n}(x_i)\geq 0\) for all \(i\in \{1,\ldots,d\}\) and \(n\in \N\,.\) The last inequality follows from \(\alpha>0\) and \(1-\max_i\{x_i\}> \varepsilon\,.\)

  Due to \eqref{corveras1} and \eqref{corveras2}, also condition \eqref{theconvthsec2} of Theorem \ref{theconvthsec} is satisfied. Then, Theorem \ref{theconvthsec} implies with condition \eqref{corconsemcopres3} of Corollary \ref{corconsemcopres} that
\begin{align*}
\psi_{S\circ (F_{1,n},\ldots,F_{d,n})}(f) =\psi_{g_n}(f)
\to \pi_f(\overline{g})
 = \pi_f(\overline{S})
\end{align*}
which proves the statement.
\end{proof}

\begin{proof}[Proof of Theorem \ref{cgoo}:]
We first show \eqref{cgoo2}: Assume that \(H\leq_{uo} H'\) and thus \(\overline{H}(x)\leq \overline{H'}(x)\) for all \(x\in \R^d\,.\) Since for all \(I=\{i_1,\ldots,i_k\}\subseteq \{1,\ldots,d\}\,,\) the \(I\)-marginal of \(f\) exists and induces a (non-negative) measure on \(\R_+^{|I|}\,,\) see Proposition \ref{cordelmof}, it follows that
\begin{align}
\int_{\R_+^{|I|}} \overline{H}(x)\de \nu_{f_I}(x) \leq \int_{\R_+^{|I|}} \overline{H'}(x)\de \nu_{f_I}(x) ~~~ \text{for all } I \subseteq\{1,\ldots,d\}\,, I\ne \emptyset\,.
\end{align}
Further, using that \(H\) and \(H'\) are grounded, it follows that \(f(0) \overline{H}(0) = f(0) \overline{H'}(0)\,.\) Hence, we obtain \(\pi_f(\overline{H})\leq \pi_f(\overline{H'})\,.\) For the reverse direction, choose for \(x\in \R_+^d\) the \(\Delta\)-monotone function \(f(y)=\1_{\{x<y\}}\,.\) Then \(\nu_f=\delta_x\) and \(\nu_{f_I}\equiv 0\) for all \(I\subsetneq \{1,\ldots,d\}\,.\) Hence, \(\pi_f(\overline{H})\leq \pi_f(\overline{H'})\) implies \(\overline{H}(x)\leq \overline{H'}(x)\,.\)\\
The proof of Statement \eqref{cgoo1} follows similarly using the inclusion-exclusion principle and the fact that \(H(x)=\overline{\overline{H}}(x)\,.\)
\end{proof}

\subsection{Proofs of the results from Appendix A}

\begin{proof}[Proof of Proposition \ref{lemfmi}:]
\eqref{lemfmi1}: Let \(x=(x_1,\ldots,x_d),y=(y_1,\ldots,y_d)\in \Xi\) such that \(x\leq y\,.\) If \(x_i=y_i\) for some \(i\in \{1,\ldots,d\}\,,\) then the statement is trivial, see Remark \ref{remindmeas}\eqref{eqremindmeas2}. So, assume that \(x_i< y_i\) for all \(i\,.\) Then we obtain
\begin{align}
\nonumber \nu_f\bigg(\bigtimes_{i=1}^d [x_i,y_i)\bigg)
&= \nu_f\bigg( \bigcup_{\eta_1\in (0,y_1-x_1)}\cdots \bigcup_{\eta_d\in (0,y_d-x_d)} \bigtimes_{i=1}^d [x_i,y_i-\eta_i]\bigg) \\
\nonumber &= \lim_{\eta_1\downarrow 0} \cdots \lim_{\eta_d \downarrow 0} \nu_f\bigg( \bigtimes_{i=1}^d  [x_i,y_i-\eta_i]\bigg)\\
\nonumber &= \lim_{\eta_1\downarrow 0} \cdots \lim_{\eta_d \downarrow 0} \lim_{\delta_i,\gamma_i\downarrow 0\atop i\in \{1,\ldots,d\}} \Delta_{a\vee(x-\delta),b\wedge(y-\eta+\gamma)}[f] \\
\label{eqprolem1}&= \Delta_{x,y}[f_-] = \Delta_{x,y} [f] = \Delta_{\varepsilon_1}^1 \cdots \Delta_{\varepsilon_d}^d f(x)\,,
\end{align}
where \(\eta=(\eta_1,\ldots,\eta_d)\,,\) \(\delta=(\delta_1,\ldots,\delta_d)\,,\) and \(\gamma=(\gamma_1,\ldots,\gamma_d)\,.\) The second equality follows from the continuity of measures.  The third equality holds by definition of \(\nu_f\,,\) see \eqref{defmind}. The fifth equality follows from the left-continuity of \(f\,,\) where \(f_-\) denotes the componentwise left-continuous version of \(f\,,\) which exists by assumption.
The last equality holds true by the definition of \(\Delta_{x,y}[f]\) in \eqref{defdvdo}.
To show the fourth equality, fix a component \(i\in \{1,\ldots,d\}\) and let \(z_j\in \Xi_j\,,\) for \(j\ne i\,.\) Then, by \eqref{defdvdo}, we obtain for the difference operator applied to the \(i\)-th component of \(f\) with \(h:=f(z_1,\ldots,z_{i-1},\cdot,z_{i+1},\ldots,z_d)\) that
\begin{align*}
\lim_{\eta_i\downarrow 0}\lim_{\delta_i,\gamma_i\downarrow 0} \Delta_{x_i-\delta_i,y_i-\eta_i+\gamma_i}[h] 
&= \lim_{\eta_i\downarrow 0} \lim_{\delta_i,\gamma_i\downarrow 0} \left[h(y_i-\eta_i+\gamma_i)-h(x_i-\delta_i)\right]\\
&= \lim_{\eta_i\downarrow 0}\lim_{\gamma_i\downarrow 0} h(y_i-\eta_i+\gamma_i)-\lim_{\delta_i\downarrow 0} h(x_i-\delta_i)\\
&= \lim_{\eta_i\downarrow 0} h_+(y_i-\eta_i) - h_-(x_i) \\
&= h_-(y_i)-h_-(x_i) = \Delta_{x_i,y_i}[h_-]
\end{align*}
where we use for the fifth equality that the left-hand limit at \(y_i\) coincides with the left-hand limit of the right-hand limits below \(y_i\,.\) \\
Statement \eqref{lemfmi2} follows similarly to \eqref{lemfmi1}.\\
\eqref{lemfmi3}: As a consequence of \eqref{lemfmi1} and \eqref{lemfmi2}, the boundary of each cuboid \(\bigtimes_{i=1}^d [x_i,y_i]\) has zero \(\nu_f\)-measure which implies the statement.
\end{proof}

\begin{proof}[Proof of Lemma \ref{lemunimifun}:]
Assume that \(f\) is measure-inducing.
If \(x,y\) with \(x\leq y\) are inner points of \(\Xi\,,\) then it follows that
\begin{align}\label{proolemunimifun1}
\nu_f([x,y])=D_{x,y}[f]=D_{x,y}[f_+]=D_{x,y}[f_-]
\end{align}
because \(\nu_f\) depends by definition of \(D_{x,y}\) in \eqref{defitlimi} only on the componentwise right-hand and left-hand limits of \(f\) which exist by definition of a measure-inducing function in \eqref{defmind}.
If at least one of \(x,y\in \Xi\) is a boundary point of \(\Xi\,,\) then \eqref{proolemunimifun1} also holds true under the additional continuity assumption on the boundary. (Note that, by \eqref{defitlimi}, \(\nu_f([x,y])\) depends on the value of \(f\) at \(x_i\) or \(y_i\) whenever \(x_i=a_i\) and \(y_i=b_i\,,\) respectively, and thus the continuity assumption cannot be omitted).
 Hence, also \(f_+\) and \(f_-\) are measure-inducing and induce, in particular, the same measure.

If \(f_+\) or \(f_-\) is measure-inducing, the statement follows similarly.
\end{proof}

\begin{proof}[Proof of Proposition \ref{propmeatraf}:]
In order to prove that \(f\circ (\phi_1^{-1},\ldots,\phi_d^{-1})\) is measure-inducing, we need to show due to Definition \ref{defmifun} that there exists a signed measure \(\eta\) on \(\bigtimes_{i=1}^d \phi_i(\Xi_i)\) such that
\begin{align*}
D_{u,v}[f\circ (\phi_1^{-1},\ldots,\phi_d^{-1})] = \lim_{\delta_i,\varepsilon_i\downarrow 0\atop i\in \{1,\ldots,d\}} \eta\left(\bigtimes_{i=1}^d [u_i-\delta_i,v_i+\varepsilon_i]\right)
\end{align*}
for all \(u=(u_1,\ldots,u_d),v=(v_1,\ldots,v_d)\in \bigtimes_{i=1}^d \phi_i(\Xi_i)\) with \(u\leq v\,.\) We show that \(\eta=\nu_f^{(\phi_1,\ldots,\phi_d)}\) is the correct choice, which implies \eqref{eqpropmeatraf}.

Define \(x:=\phi^{-1}(u):=(\phi_1^{-1}(u_1),\ldots,\phi_d^{-1}(u_d))\) and \(y:=\phi^{-1}(v):=(\phi_1^{-1}(v_1),\ldots,\phi_d^{-1}(v_d))\) for such \(u\) and \(v\,.\)  Then, it follows that
\begin{align*}
D_{u,v}[f\circ(\phi_1^{-1},\ldots,\phi_d^{-1})] &= D_{\phi^{-1}(u),\phi^{-1}(v)} [f]\\
&= D_{x,y}[f] \\
&= \nu_f\left(\bigcap_{\delta_i,\varepsilon_i>0 \atop i\in \{1,\ldots,d\}} \bigtimes_{i=1}^d[x_i-\delta_i,y_i+\varepsilon_i]\right) \\
&= \nu_f\left(\bigcap_{\delta_i,\varepsilon_i>0 \atop i\in \{1,\ldots,d\}} \bigtimes_{i=1}^d [\phi_i^{-1}(u_i-\delta_i),\phi_i^{-1}(y_i+\varepsilon_i)]\right)\\
&= \nu_f^{(\phi_1,\ldots,\phi_d)} \left(\bigcap_{\delta_i,\varepsilon_i>0\atop i\in \{1,\ldots,d\}} \bigtimes_{i=1}^d[u_i-\delta_i,y_i-\varepsilon_i]\right)\\
&= \lim_{\delta_i,\varepsilon_i\downarrow 0\atop i\in \{1,\ldots,d\}} \nu_f^{(\phi_1,\ldots,\phi_d)} \left(\bigtimes_{i=1}^d [u_i-\delta_i,y_i-\varepsilon_i]\right)
\,,
\end{align*}
where first equality follow from the definition of the difference operator \(D_{u,v}\) using that the transformations \(\phi_1^{-1},\ldots,\phi_d^{-1}\) are continuous and strictly increasing. For the third equality, we apply that \(f\) is measure-inducing.  For the fourth equality, we again use that \(\phi_1^{-1},\ldots,\phi_d^{-1}\) are continuous and strictly increasing. The fifth equality follows from the definition of the image of a (signed) measure, and the last equality holds true by the definition of the iterated limit in \eqref{eqappme}.
\end{proof}

\begin{proof}[Proof of Lemma \ref{lemmargdisgrou}:]
Let \(I=\{i_1,\ldots,i_k\}\subseteq \{1,\ldots,d\}\) and \(x_j\in \Xi_j\,,\) for \(j\in I\,.\) Define \(y_j:=x_j\) for \(j\in I\) and \(y_j:=b_j=\sup \Xi_j\) for \(j\ne I\,.\) Then, it follows with \(a_i=\inf \Xi_i\,,\) \(i=1,\ldots,d\,,\) that
\begin{align*}
\nu_{f_+}^I\left(\bigtimes_{i\in I}(a_{i},x_{i}]\right)
= \nu_{f_+}\left( \bigtimes_{i=1}^d ((a_i,y_i]\cap \Xi_i)\right) &=  \lim_{x_i\uparrow b_i\atop i\ne I} f_+(x_1,\ldots,x_d)\\
&= f_+^I(x_{i_1},\ldots,x_{i_k}) =\nu_{f_+^I}\left(\bigtimes_{i\in I}(a_{i},x_{i}]\right) 
\end{align*}
where the first equality follows from the definition of the \(I\)-marginal signed measure.
For the second equality, we apply Proposition \ref{lemfmi}\eqref{lemfmi2} and use continuity condition \eqref{contboun2} which holds by assumption for \(f\) and thus also for \(f_+\,.\)
The third equality holds by definition of the upper \(I\)-marginal in \eqref{defqmarg}.
The last equality follows from Proposition \ref{lemfmi}\eqref{lemfmi1} using that \(f\) and thus \(f_+\) as well as \(f_+^I\) are grounded.
Applying Lemma \ref{lemunimifun}, we obtain that \(\nu_{f^I}=\nu_{f_+^I}=\nu_{f_+}^I = \nu_f^I\,.\) This yields \eqref{lemmargdisgrou2} and \(f^I\in \cF_{mi}^{c,l}(\Xi_I)\) which implies \eqref{lemmargdisgrou1} noting that \(f_I\) induces the null-measure for all \(I\subsetneq \{1,\ldots,d\}\) because \(f\) is grounded.
%
\end{proof}

\begin{proof}[Proof of Lemma \ref{lemFfgr}:]
Assume w.l.o.g. that \(x_1\to a_1\,.\) Then, for \(x_j\in \Xi_j\,,\) \(j\in \{2,\ldots,d\}\,,\) it follows that
\begin{align*}
\lim_{x_1\downarrow a_1} F_f(x_1,x_2,\ldots,x_d)&= \lim_{x_1\downarrow a_1}  f(x_1,x_2,\ldots,x_d) - f_{\{2,\ldots,d\}}(x_2,\ldots,x_d) \\
&~~~- \sum_{I\subsetneq\{2,\ldots,d\}\atop I=\{i_1,\ldots,i_k\}} (-1)^{d-|I|} \lim_{x_1\downarrow a_1}  f_{I\cup \{1\}}(x_1,x_{i_1},\ldots,x_{i_k}) \\
&~~~- \sum_{I\subsetneq\{2,\ldots,d\}\atop I=\{i_1,\ldots,i_k\}} (-1)^{d-|I|-1} f_I(x_{i_1},\ldots,x_{i_k})\\
&=0\,,
\end{align*}
where the last equality follows from the definition of the lower \(I\)-marginal, i.e., \(\lim_{x_1\downarrow a_1} f(x_1,x_2,\ldots,x_d)=f_{\{2,\ldots,d\}}(x_2,\ldots,x_d)\) and \(\lim_{x_1\downarrow a_1} f_{I\cup \{1\}}(x_1,x_{i_1},\ldots,x_{i_k})= f_{I}(x_{i_1},\ldots,x_{i_k})\) for all \(I\subsetneq \{2,\ldots,d\}\,.\)
\end{proof}

\begin{proof}[Proof of Theorem \ref{propbhkv}:]
Assume \eqref{propbhkv1}. Let \(I\subseteq \{1,\ldots,d\}\,,\) \(I\ne \emptyset\,.\)
Since \(f\) has bounded Hardy-Krause variation, also \(f_I\) has bounded Hardy-Krause variation. Hence, also \(F^{(f)_I}\) defined by \eqref{defFf} has bounded Hardy-Krause variation. Due to \cite[Theorem 3a]{Aistleitner-2015}, \(F^{(f)_I}\) generates a signed measure \(\nu\) by \(\nu([0,x])=F^{(f)_I}(x)\,,\) \(x\in [0,1]^d\,,\) using that \(f\) is right-continuous.
Hence, \((f)_I\) induces a signed measure \(\nu\) on \([0,1]^{|I|}\) by \(\nu([0,x])=\Delta_{0,x}[(f)_I]\,,\) which implies by Lemma \ref{lemunimifun} that \(f_I\) is measure-inducing. This implies \eqref{propbhkv2}.
Since, by \cite[Lemma 2]{Aistleitner-2015}, \(f\) also has bounded Hardy-Krause variation anchored at \((1,\ldots,1)\,,\) it follows similarly that \(f^I\) is measure-inducing, which implies \eqref{propbhkv3}.

If \eqref{propbhkv2} or \eqref{propbhkv3} holds then \eqref{propbhkv1} follows with \cite[Theorem 3b]{Aistleitner-2015} and \cite[Lemma 2]{Aistleitner-2015} using that \(f_I\) respectively \(f^I\) is measure-inducing and right-continuous for all \(I\subseteq\{1,\ldots,d\}\,,\) \(I\ne \emptyset\,.\)
\end{proof}

\begin{proof}[Proof of Theorem \ref{thecfmicl}:]
\eqref{thecfmicl1} \(\Longrightarrow\) \eqref{thecfmicl2}: By Proposition \ref{propcharmif}, there exist for all \(I=\{i_1,\ldots,i_k\}\subseteq\{1,\ldots,d\}\,,\) \( I\ne \emptyset\,,\) uniquely determined right-continuous, \(\Delta\)-monotone, and grounded functions \(g_{(I)},h_{(I)}\colon \R_+^{k}\to \R\,,\) at least one of them is bounded, such that 
\begin{align*}
g_{(I)}(0)&=h_{(I)}(0)=0\,, &&\\
 F^{f_I}(x) &=g_{(I)}(x)-h_{(I)}(x) && \text{for all } x\in \R_+^k\,,\\
 \Var(F^{f_I}|_{[0,x]})&=\Var(g|_{[0,x]})+\Var(h|_{[0,x]})<\infty~&&\text{for all } x\in \R_+^k\,.
\end{align*}
Since, by definition of \(F^{f_I}\,,\) it holds that
\begin{align*}
F^{f_I}(x_{i_1},\ldots,x_{i_k}) = f_I(x_{i_1},\ldots,x_{i_k})- \sum_{J\subsetneq I\atop J=\{j_1,\ldots,j_m\}} (-1)^{|I|-|J|-1} f_J(x_{j_1},\ldots,x_{j_m})\,,
\end{align*}
it follows for all \(x=(x_1,\ldots,x_d)\in \R_+^d\) that
\begin{align*}
f(x)&= F^f(x) + \sum_{I\subsetneq \{1,\ldots,d\}\atop I=\{i_1,\ldots,i_k\}} (-1)^{d-|I|-1} f_I(x_{i_1},\ldots,x_{i_k})\\
&= F^f(x) + \sum_{I\subsetneq \{1,\ldots,d\}\atop I=\{i_1,\ldots,i_k\}} (-1)^{d-|I|-1} \Big( F^{f_I}(x_{i_1},\ldots,x_{i_k})+\sum_{J\subsetneq I\atop J=\{j_1,\ldots,j_l\}} (-1)^{|I|-|J|-1} f_J(x_{j_1},\ldots,x_{j_l})\Big)\\
&= f_\emptyset+ \sum_{I\subseteq \{1,\ldots,d\}\atop {I=\{i_1,\ldots,i_k\}\atop I\ne \emptyset}} F^{f_I}(x_{i_1},\ldots,x_{i_k})\\
&=  f(0)+ \sum_{k=1}^d \sum_{I\subseteq \{1,\ldots,d\}\atop I=\{i_1,\ldots,i_k\}} \left(g_{(I)}(x_{i_1},\ldots,x_{i_k}) - h_{(I)}(x_{i_1},\ldots,x_{i_k})\right)\,.
\end{align*}
This proves \eqref{thecfmicl2}.\\
\eqref{thecfmicl2} \(\Longrightarrow\) \eqref{thecfmicl1}: Since \(g_{(I)}\) and \(h_{(I)}\) are right-continuous for all \(I\subseteq\{1,\ldots,d\}\,,\) \(I\ne \emptyset\,,\) also \(f\) defined by \eqref{eqthecfmicl2} is right-continuous. Hence, \(f\) fulfils the continuity condition \eqref{contboun1}. Note that continuity condition \eqref{contboun2} is trivially satisfied. Further, since \(g_{(I)}\) and \(h_{(I)}\) are \(\Delta\)-monotone, it follows that all componentwise right- and left-hand limits of \(f_I\) exist for all  \(I\subseteq \{1,\ldots,d\}\,,\) \(I\ne \emptyset\,.\) It remains to show that \(f_I\) is measure-inducing for all \(I\subseteq \{1,\ldots,d\}\,,\) \(I\ne \emptyset\,.\) Hence, fix such an index set \(I\,.\) Due to \eqref{eqthecfmicl2}, \(f_I\) is given by
\begin{align*}
f_I(x_{i_1},\ldots,x_{i_k})= f(0) + \sum_{k=1}^{|I|} \sum_{J\subseteq I\atop J=\{j_1,\ldots,j_m\}} \left(g_{(J)}(x_{j_1},\ldots,x_{j_m}) - h_{(J)}(x_{j_1},\ldots,x_{j_m})\right)
\end{align*}
applying that \(g_{(L)}\) and \(h_{(L)}\) are grounded for \(L\) with \(I\subsetneq L\subseteq\{1,\ldots,d\}\,.\) This implies that
\begin{align*}
\Delta_{y,z}[f_I]=\Delta_{y,z}[g_{(I)}]-\Delta_{y,z}[h_{(I)}] ~~~ \text{for all } y,z\in \R_+^d  \text{ such that } y\leq z\,.
\end{align*}
Then, the equivalence of \eqref{propcharmif5} and \eqref{propcharmif1} in Proposition \ref{propcharmif} yields that \(f_I\) is measure-inducing using that \(g_{(I)}\) and \(h_{(I)}\) are \(\Delta\)-monotone and thus \(|I|\)-increasing and that at least one of \(g_{(I)}\) and \(h_{(I)}\) is bounded.
\end{proof}

For the proof of Proposition \ref{propcharmif}, we make use of the following lemma.

\begin{lemma}\label{prophjd}
For a function \(f\colon \R_+^d \to \R\) with \(f(0)=0\,,\) the following statements are equivalent.
\begin{enumerate}[(i)]
\item \label{prophjd1} \(f\) generates a signed measure \(\nu\) on \(\R_+^d\) by
\begin{align}\label{defmgfn}
\nu([0,x])=f(x)~~~ \text{for all } x\in \R_+^d\,,
\end{align}
\item \label{prophjd2} there exist uniquely determined right-continuous, \(\Delta\)-monotone functions \(f_\Delta^+,f_\Delta^- \colon \R_+^d \to \R\,,\) at least one of which is bounded, such that \(f_\Delta^+(0)= f_\Delta^-(0) = 0\,,\)
\begin{align*}
f(x) &= f_\Delta^+(x)-f_\Delta^-(x)~~~ \text{and}\\
\Var(f|_{[0,x]}) &= \Var(f_\Delta^+|_{[0,x]}) + \Var(f_\Delta^-|_{[0,x]}) < \infty
\end{align*}
for all \(x\in \R_+^d\,.\)
\end{enumerate}
\end{lemma}

\begin{proof}[Proof of Lemma \ref{prophjd}:]
\eqref{prophjd1} \(\Longrightarrow\) \eqref{prophjd2}: For \(y\in \R_+^d\,,\) consider the function \(f_y:=f_{|[0,y]}\,.\) Denote by \(\nu_y\) the signed measure \(\nu\) restricted to \([0,y]\,,\) i.e., \(\nu_y(A):=\nu(A\cap [0,y])\) for all Borel sets \(A\subset \R_+^d\,.\) Then it holds for all \(x\in [0,y]\) that
\begin{align*}
f_y(x)= f(x) = \nu([0,x])=\nu_y([0,x])\,.
\end{align*}
Hence, similar to \eqref{eqremindmeas1}, \(\nu_y\) is a finite signed measure.

Due to \cite[Theorem 3b]{Aistleitner-2015} and a transformation by Proposition \ref{propmeatraf}, there exist uniquely determined \(\Delta\)-monotone functions \(f_y^+,f_y^-\colon [0,y]\to \R\) with \(f_y^+(0)=f_y^-(0)=0\) such that
\begin{align}
\nonumber f_y(x) &= f_y^+(x)-f_y^-(x) ~~~ \text{for all } x\in [0,y]\,,\\
\label{prprophjd} f_y^+(x)&= \nu_y^+([0,x]) ~~~\text{and} ~~~ f_y^-(x)=\nu_y^-([0,x])~~~\text{for all } x\in [0,y]\,,\\
\nonumber \Var(f_y)&= \Var(f_y^+)+\Var(f_y^-)<\infty
\end{align}
noting that \(f_y(0)=f(0)=0\) and \((\nu_y^+,\nu_y^-)\) is the Jordan-decomposition of \(\nu_y\,.\) Since \(f\) and thus \(f_y\) are right-continuous, it follows from \cite[Proof of Theorem 3a]{Aistleitner-2015} that also \(f_y^+\) and \(f_y^-\) are right-continuous.

Now, define \(f_\Delta^+,f_\Delta^-\colon \R_+^d \to \R\) by \(f_\Delta^+(x):=f_x^+(x)\) and \(f_\Delta^-(x):= f_x^-(x)\,.\) Then \(f_\Delta^+\) and \(f_\Delta^-\) are right-continuous and \(\Delta\)-monotone with \(f_\Delta^+(0)=f_\Delta^-(0)=0\) and it holds that
\begin{align*}
f(x)&=f_x(x)=f_\Delta^+(x)-f_\Delta^-(x)\,,\\
\Var(f|_{[0,x]}) &= \Var(f_\Delta^+|_{[0,x]}) + \Var(f_\Delta^-|_{[0,x]}) < \infty
\end{align*}
for all \(x\in \R_+^d\,.\) The uniqueness of \(f_\Delta^+\) and \(f_\Delta^-\) follows from the uniqueness of \(f_y^+\) and \(f_y^-\) together with the property that \(f_\Delta^+|_{[0,y]}=f_y^+\) and \(f_\Delta^-|_{[0,y]}=f_y^-\) for all \(y\in \R_+^d\,.\) Since at least one of the measures \(\nu^+\) and \(\nu^-\) is finite, it follows from \eqref{prprophjd} that at least one of the families \((f_y^+)_{y\geq 0}\) and \((f_y^-)_{y\geq 0}\) is bounded. This implies that at least one of the functions \(f_\Delta^+\) and \(f_\Delta^-\) is bounded.

\eqref{prophjd2} \(\Longrightarrow\) \eqref{prophjd1}: For \(y\in \R_+^d\,,\) the function \(f_y:=f|_{[0,y]}\) is right-continuous and has bounded Hardy-Krause variation. Due to \cite[Theorem 3a]{Aistleitner-2015} and a transformation by Proposition \ref{propmeatraf}, there exists a unique signed measure \(\nu_y\) such that
\begin{align}
\nu_y([0,x])&=f_y(x) \,,\\
\nu_y^+([0,x])&=f_y^+(x) ~~~\text{and}~~~ \nu_y^-([0,x])=f_y^-(x)
\end{align}
for all \(x\in [0,y]\,,\) where \((\nu_y^+,\nu_y^-)\) is the Jordan decomposition of \(\nu_y\) and where \(f_y=f_y^+-f_y^-\) with \(f_y^+:=f_\Delta^+|_{[0,y]}\) and \(f_y^-:=f_\Delta^-|_{[0,y]}\) is the uniquely determined so-called Jordan decomposition of the function \(f\,.\) For all \(x\in \R_+^d\,,\) define 
\begin{align*}
\nu^+([0,x])&:=\nu_x^+([0,x])=f_x^+(x)=f_\Delta^+(x) ~~~\text{and} \\
\nu^-([0,x])&:=\nu_x^-([0,x])=f_x^-(x)=f_\Delta^-(x) \,.
\end{align*}
Since \(f_\Delta^+\) and \(f_\Delta^-\) are by assumption \(\Delta\)-monotone, it follows that \(\nu^+\) and \(\nu^-\) define (non-negative) measures on \(\R_+^d\) where at least one of which is finite because at least one of the functions \(f_\Delta^+\) and \(f_\Delta^-\) is bounded. Hence, \(\nu:=\nu^+-\nu^-\) defines a signed measure with the property that
\begin{align*}
\nu([0,x])=\nu^+([0,x])-\nu^-([0,x])=f_\Delta^+(x)-f_\Delta^-(x) = f(x)
\end{align*}
for all \(x\in \R_+^d\,,\) which concludes the proof.
\end{proof}

\begin{proof}[Proof of Proposition \ref{propcharmif}:]
\eqref{propcharmif1} \(\Longrightarrow\) \eqref{propcharmif2}: Since \(f\) is measure-inducing, it follows by \eqref{defitlimi} for \(0=(0,\ldots,0)\in \R_+^d\)
for all \(x=(x_1,\ldots,x_d)\in \R_+^d\) that
\begin{align}\label{eqproopropcharmif1}
\begin{split}
\nu_f([0,x])=D_{0,x}[f]=\Delta_{0,x}[f]=f(x)-\sum_{k=0}^{d-1}\sum_{I\subsetneq \{1,\ldots,d\}\atop I=\{i_1,\ldots,i_k\}} (-1)^{d-k-1} f_I(x_{i_1},\ldots,x_{i_k}) = F_f(x)\,,
\end{split}
\end{align}
where we applied for the second equality that \(f\) is right-continuous. Hence, \(F^f\) 
generates the signed measure \(\nu=\nu_f\) by \(\nu [0,x]=F^f(x)\,,\) \(x\in \R_+^d\,.\) \\
\eqref{propcharmif2} \(\Longrightarrow\) \eqref{propcharmif3}: Assume that \(F^f\) generates a signed measure by \(\nu [0,x]=F^f(x)\,,\) \(x\in \R_+^d\,.\)
By Lemma \ref{lemFfgr}, \(F^f\) is grounded. Since \(f\) is right-continuous, also \(F^f\) is right-continuous.
Due to Lemma \ref{prophjd}, there exist uniquely determined, right-continuous, \(\Delta\)-monotone functions \(g,h\colon \R_+^d\to \R\) such that at least one of which is finite and \eqref{eqpropcharmif1}, \eqref{eqpropcharmif2}, and \eqref{eqpropcharmif3} are satisfied.
Since \(F^f\) is grounded, for all \(I\subsetneq\{1,\ldots,d\}\) the Vitali variation of the lower \(I\)-marginal of \(F^f\) and thus, by \eqref{eqpropcharmif3}, also the Vitali variation of \(g_I\) and \(h_I\) vanish. This implies together with \eqref{eqpropcharmif2} that \(g_I\equiv 0\) and \(h_I\equiv 0\,.\) Hence, \(g\) and \(h\) are grounded.
Equation \eqref{eqpropcharmif} follows from \eqref{eqpropcharmif2} and the definition of \(F^f\) because
\begin{align*}
\Delta_{x,y}[f]=\Delta_{x,y}[F^f]=\Delta_{x,y}[g]-\Delta_{x,y}[h]\,,
\end{align*}
using for the first equality that \(\Delta_{x,y}[f]=0\) whenever \(x_i=y_i=0\) for \(i\in \{1,\ldots,d\}\setminus I\,,\) see Remark \ref{remindmeas}\eqref{eqremindmeas2}. The second equality follows from \eqref{eqpropcharmif2}.\\
\eqref{propcharmif3} \(\Longrightarrow\) \eqref{propcharmif5}: The statement is trivial because \(\Delta\)-monotone functions are \(d\)-increasing.\\
\eqref{propcharmif5} \(\Longrightarrow\) \eqref{propcharmif1}: The functions \(F^g\) and \(F^h\) defined by \eqref{defFf} are \(d\)-increasing and grounded, and at least one of them is bounded. Hence, \(F^g\) and \(F^h\) induce (non-negative) measures \(\eta_g\) and \(\eta_h\,,\) at least one which is finite. It follows for all \(x,y\in \R_+^d\) with \(x\leq y\) that
\begin{align*}
D_{x,y}[f]=D_{x,y}[g]-D_{x,y}[h] = D_{x,y}[F^g]-D_{x,y}[F^h]=\eta_g([x,y])-\eta_h([x,y])\,.
\end{align*}
Thus \(f\) induces by \eqref{defmind} a signed measure \(\nu:=\eta_g-\eta_h\,.\) 
\end{proof}

\begin{proof}[Proof of Lemma \ref{lemsurvmarg}:]
Assume w.l.o.g. that \(I=\{1,\ldots,k\}\) for some \(k\in \{1,\ldots,d\}\,.\) Let \((x_1,\ldots,x_d)\in \Xi\,.\) Then the definition of the upper \(I\)-marginal and the definition of the survival function imply
\begin{align*}
f^I(x_1,\ldots,x_k) &= f(x_1,\ldots,x_k,b_{k+1},\ldots,b_d)~~~\text{and}\\
\overline{f^I}(x_1,\ldots,x_k)&= \sum_{J\subseteq \{1,\ldots,k\}\atop {y_i=b_i ~\text{if } i\in J \atop y_i=x_i ~\text{if } i\notin J}} (-1)^{k-|J|} f^I(y_1,\ldots,y_k)\,.
\end{align*}
It holds that
\begin{align*}
\overline{f}(x_1,\ldots,x_d) &= \sum_{J\subseteq \{1,\ldots,d\} \atop {y_i=b_i ~\text{if } i\in J \atop y_i=x_i ~\text{if } i\notin J}} (-1)^{d-|J|} f(y_1,\ldots,y_d) \\
&= \sum_{J'\subseteq \{1,\ldots,k\}}\sum_{J''\subseteq \{k+1,\ldots,d\}} (-1)^{d-|J'|-|J''|} f(z_1,\ldots,z_d)
\end{align*}
for \(z_i=b_i\) if \(i\in J'\cup J''\) and \(z_i=x_i\) else. This implies for the lower \(I\)-marginal of \(\overline{f}\) that
\begin{align*}
\overline{f}_I(x_1,\ldots,x_k) &= \sum_{J'\subseteq \{1,\ldots,k\} \atop {y_i=b_i ~\text{if } i\in J' \atop y_i=x_i ~\text{if } i\notin J'}} (-1)^{d-|J'|-(d-k)} f(y_1,\ldots,y_k,b_{k+1},\ldots,b_d) \\
&= \sum_{J\subseteq \{1,\ldots,k\} \atop {y_i=b_i ~\text{if } i\in J \atop y_i=x_i ~\text{if } i\notin J}} (-1)^{k-|J|} f^I(y_1,\ldots,y_k)\\
&= \overline{f^I}(x_1,\ldots,x_k)
\end{align*}
where we applied for the first equality the definition of the lower \(I\)-marginal as well as groundedness of \(f\,.\)
\end{proof}

\begin{proof}[Proof of Corollary \ref{cordelmof}:]
If \(f\) is \(\Delta\)-monotone/-antitone, it is componentwise increasing/decreasing. Hence, its right-continuous version \(f_+\) exists and is also \(\Delta\)-monotone/-antitone. Since for \(I=\{i_1,\ldots,i_k\}\subseteq \{1,\ldots,d\}\,,\) \((f_+)_I\) and \((-1)^k (f_+)_I\,,\) respectively, is \(k\)-increasing, \((f_+)_I\) is by Proposition \ref{propcharmif} measure-inducing and, in particular, induces a non-negative measure.
Applying the continuity condition \eqref{contboun1}, also \(f_I\) is by Lemma \ref{lemunimifun} measure-inducing. This yields that \(f\in  \cF_{mi}^{c,l}(\R_+^d)\,.\)
\end{proof}

\begin{proof}[Proof of Proposition \ref{propcdfun}:]
For \(I=\{i_1,\ldots,i_k\}\subseteq \{1,\ldots,d\}\,,\) it holds that
\begin{align*}
\V(f_I) \leq \int_{[0,1]^k} \left| \frac{\partial^k f_I(x)}{\partial x_1 \cdots \partial x_k} \right| \de x \,,
\end{align*}
see \cite[Proposition 13]{Owen-2005}, where the integral is finite due to the continuity of the derivatives. This implies by definition of the Hardy-Krause variation that
\begin{align*}
\Var(f) = \sum_{I\subseteq\{1,\ldots,d\}\atop I\ne \emptyset} \V(f) < \infty\,.
\end{align*}
Hence, the statement follows from Theorem \ref{propbhkv}.
\end{proof}



\nocite{*}

%

%





%
%
%

\end{document}